\documentclass{article}
\usepackage[cp1251]{inputenc}
\usepackage[russian]{babel}
\usepackage{amssymb,amsfonts,amsmath,amsthm,amscd}
\usepackage{graphicx}
\usepackage{enumerate}
\usepackage{soul}
\usepackage[matrix,arrow,curve]{xy}
\usepackage{longtable}
\usepackage{array}
\usepackage{bm}

\newdimen\symskip
\newdimen\defskip
\newdimen\parind
\parind=\parindent
\newdimen\leftmarge
\newdimen\theoremshape
\topsep\defskip \righthyphenmin=2

\makeatletter
\newcommand*{\clei}{\nobreak\hskip\z@skip}
\newcommand{\?}{\,\nobreak\hskip0pt}
\renewcommand{\"}{''}
\renewcommand{\:}{\textup{:}}
\renewcommand{\~}{\textup{;}}
\DeclareRobustCommand*{\т}{~\textemdash{} }
\DeclareRobustCommand*{\д}{\clei\hbox{-}\clei}
\newcommand{\no}{}
\renewcommand{\@listI}{\settowidth\labelwidth{\labheadi{\no}}\listipar{\parind}{\labelwidth}}
\newcommand{\listivpar}{\topsep\defskip\partopsep0pt\parsep-\parskip\itemsep0.5\topsep}
\newcommand{\listipar}[2]{\rightmargin0pt\leftmargin#1\labelsep#1\advance\labelsep-#2\itemindent0pt\listivpar}
\renewcommand{\@listii}{\settowidth\labelwidth{\labheadii{\@roman{\no}}}\listiipar{\parind}{\labelwidth}}
\newcommand{\listiivpar}{\topsep0.5\defskip\partopsep0pt\parsep-\parskip\itemsep0.5\topsep}
\newcommand{\listiipar}[2]{\rightmargin0pt\leftmargin#1\labelsep#1\advance\labelsep-#2\itemindent0pt\listiivpar}
\def\thempfn{\ifcase\value{footnote}1\or *\or **\or ***\else\@ctrerr\fi}
\renewcommand\footnoterule{%
  \kern-3\p@
  \hrule\@width1in
  \kern2.6\p@}
\makeatother

\makeatletter

\renewcommand{\@biblabel}[1]{[#1]}
\renewenvironment{thebibliography}[1]
     {\renewcommand{\refname}{Литература}%
      \section*{\refname}%
      \@mkboth{\MakeUppercase\refname}{\MakeUppercase\refname}%
      \list{\@biblabel{\@arabic\c@enumiv}}%
           {\itemsep\baselineskip
            \leftmargin\parind
            \settowidth\labelwidth{\@biblabel{#1}}%
            \labelsep\parind\advance\labelsep-\labelwidth
            \@openbib@code
            \usecounter{enumiv}%
            \let\p@enumiv\@empty
            \renewcommand\theenumiv{\@arabic\c@enumiv}}%
      \sloppy
      \clubpenalty4000
      \@clubpenalty\clubpenalty
      \widowpenalty4000%
      \sfcode`\.\@m}
     {\def\@noitemerr
       {\@latex@warning{Empty `thebibliography' environment}}%
      \endlist}

\def\@maketitle{%
  \newpage
  \vskip0.5em%
  УДК \udk%
  \vskip0.5em%
  MSC \msc%
  \vskip1em%
  \begin{center}\bf%
  \let\footnote\thanks%
   {\Large\@author\par}%
   \vskip1.5em%
   {\LARGE\@title\par}%
   \vskip1em%
   {\large\@date}%
  \end{center}%
  \par
  \vskip1.5em}

\def\@title{\@latex@warning@no@line{No \noexpand\title given}}

\makeatother

\makeatletter

\renewcommand\sectionmark[1]{%
 \markright{%
  \ifnum \c@secnumdepth >\z@
   \thesection. \ %
  \fi
 #1}}%

\renewcommand{\section}{\@startsection{section}{1}{0pt}%
{5.5ex plus .5ex minus .2ex}{1.5ex plus .3ex}%
{\center\normalfont\Large\bfseries\sffamily\bom}}
\renewcommand{\subsection}{\@startsection{subsection}{2}{0pt}%
{4.5ex plus .4ex minus .2ex}{0.75ex plus .2ex}%
{\center\normalfont\large\bfseries\sffamily\bom}}
\renewcommand{\subsubsection}{\@startsection{subsubsection}{3}{0pt}%
{2.5ex plus .5ex minus .2ex}{1ex plus .2ex}%
{\center\normalfont\bfseries\sffamily\bom}}

\newcommand{\Ss}{\textup{\S\,}}
\newcommand{\Sss}{\textup{\S\S\,}}
\def\@postskip@{\hskip.5em\relax}

\def\postsection{.\@postskip@}

\def\postsubsection{.\@postskip@}

\def\postsubsubsection{.\@postskip@}

\def\postparagraph{.\@postskip@}

\def\postsubparagraph{.\@postskip@}

\def\@seccntformat#1{\csname pre#1\endcsname\csname the#1\endcsname\csname post#1\endcsname}

\makeatother

\renewcommand{\thesection}{\textup{\arabic{section}}}

\newcommand{\parr}{\par\addvspace{\defskip}}
\newcommand{\theo}[2]{\newtheorem{#1}{#2}[section]}
\newcommand{\deff}[2]{\newenvironment{#1}{\parr\textbf{#2.}}{\parr}}
\theo{cas}{Случай} \theo{problem}{Проблема} \theo{theorem}{Теорема}
\theo{lemma}{Лемма} \theo{prop}{Предложение} \theo{stm}{Утверждение}
\theo{imp}{Следствие} \theo{ex}{Пример}
\deff{df}{Определение}
\deff{note}{Замечание}
\deff{denote}{Обозначение}
\deff{denotes}{Обозначения}
\deff{hint}{Указание}
\deff{answer}{Ответ\:}

\makeatletter
\def\@begintheorem#1#2[#3]{%
  \deferred@thm@head{\the\thm@headfont \thm@indent
    \@ifempty{#1}{\let\thmname\@gobble}{\let\thmname\@iden}%
    \@ifempty{#2}{\let\thmnumber\@gobble}{\let\thmnumber\@iden}%
    \@ifempty{#3}{\let\thmnote\@gobble}{\let\thmnote\@iden}%
    \thm@notefont{\bfseries\upshape}%
    \indent%
    \thm@swap\swappedhead\thmhead{#1}{#2}{#3}%
    \the\thm@headpunct
    \thmheadnl 
    \hskip\thm@headsep
  }%
  \ignorespaces}
\renewenvironment{proof}{\setcounter{cas}{0}\parr\pushQED{\qed}\normalfont$\square\quad$}{\setcounter{cas}{0}\popQED\@endpefalse\parr}

\makeatother

\newcommand{\labheadi}[1]{\textup{#1)}}
\newcommand{\labheadii}[1]{\textup{(#1)}}

\renewcommand{\theenumi}{\labhi{enumi}}

\newenvironment{nums}[1]{\renewcommand{\no}{#1}\begin{enumerate}}{\end{enumerate}}
\newcommand{\eqn}[1]{\begin{equation}#1\end{equation}}
\newcommand{\equ}[1]{\begin{equation*}#1\end{equation*}}

\makeatletter

\def\LT@makecaption#1#2#3{%
  \LT@mcol\LT@cols c{\hbox to\z@{\hss\parbox[t]\LTcapwidth{%
    \sbox\@tempboxa{#1{#2. }#3}%
    \ifdim\wd\@tempboxa>\hsize
      #1{#2. }#3%
    \else
      \hbox to\hsize{\hfil\box\@tempboxa\hfil}%
    \fi
    \endgraf\vskip\baselineskip}%
  \hss}}}

\@addtoreset{equation}{section} \@addtoreset{footnote}{section}

\newenvironment{casks}{%
  \matrix@check\casks\env@casks
}{%
  \endarray\right.%
}
\def\env@casks{%
  \let\@ifnextchar\new@ifnextchar
  \left\lbrack
  \def\arraystretch{1.2}%
  \array{@{}l@{\quad}l@{}}%
}

\newcounter{numt}
\newcounter{col}
\newcounter{coll}

\makeatother

\renewcommand{\ge}{\geqslant}
\renewcommand{\le}{\leqslant}
\newcommand{\fa}{\,\forall\,}

\newcommand{\bes}{\infty}
\newcommand{\es}{\varnothing}

\newcommand{\subs}{\subset}
\newcommand{\sups}{\supset}

\newcommand{\sm}{\setminus}

\newcommand{\swo}{\mathbin{\triangle}}
\newcommand{\cln}{\colon}
\newcommand{\nl}{\lhd}

\newcommand{\Ra}{\Rightarrow}

\newcommand{\dv}{\smash{\mskip3mu\lower1pt\hbox{\vdots}\mskip3mu}}

\newcommand{\ol}{\overline}

\newcommand{\wt}{\widetilde}

\newcommand{\sst}[1]{\substack{#1}}

\newcommand{\suml}[2]{\sum\limits_{{#1}}^{{#2}}}
\newcommand{\sums}[1]{\sum\limits_{{#1}}}

\newcommand{\oplusl}[2]{\bigoplus\limits_{{#1}}^{{#2}}}
\newcommand{\opluss}[1]{\bigoplus\limits_{{#1}}^{}}

\renewcommand{\caps}[1]{\bigcap\limits_{{#1}}}
\newcommand{\cupl}[2]{\bigcup\limits_{{#1}}^{{#2}}}
\newcommand{\cups}[1]{\bigcup\limits_{{#1}}}
\newcommand{\sqcupl}[2]{\bigsqcup\limits_{{#1}}^{{#2}}}

\newcommand*{\bw}[1]{#1\nobreak\discretionary{}{\hbox{$\mathsurround=0pt #1$}}{}}
\newcommand{\sco}{,\ldots,}
\newcommand{\spl}{\bw+\ldots\bw+}

\newcommand{\seq}{\bw=\ldots\bw=}

\newcommand{\sop}{\bw\oplus\ldots\bw\oplus}

\newcommand{\sti}{\bw\times\ldots\bw\times}

\newcommand{\ha}[1]{\left\langle#1\right\rangle}
\newcommand{\ba}[1]{\bigl\langle#1\bigr\rangle}

\newcommand{\br}[1]{\bigl(#1\bigr)}
\newcommand{\Br}[1]{\Bigl(#1\Bigr)}

\newcommand{\ter}[1]{\textup{(}#1\textup{)}}
\newcommand{\bgm}[1]{\bigl|#1\bigr|}
\newcommand{\Bm}[1]{\Bigl|#1\Bigr|}
\newcommand{\hn}[1]{\left\|#1\right\|}
\newcommand{\hnn}[1]{\|#1\|}
\newcommand{\bn}[1]{\bigl\|#1\bigr\|}

\newcommand{\bs}[1]{\bigl[#1\bigr]}

\newcommand{\bc}[1]{\bigl\{#1\bigr\}}
\newcommand{\BC}[1]{\Bigl\{#1\Bigr\}}

\newcommand{\bb}{\bigm/}

\newcommand{\mbb}{\mathbb}
\newcommand{\mbf}{\mathbf}
\newcommand{\mcl}{\mathcal}
\newcommand{\mfr}{\mathfrak}

\newcommand{\R}{\mbb{R}}

\newcommand{\Z}{\mbb{Z}}
\newcommand{\N}{\mbb{N}}
\newcommand{\T}{\mbb{T}}
\newcommand{\F}{\mbb{F}}
\newcommand{\Cbb}{\mbb{C}}
\newcommand{\Hbb}{\mbb{H}}

\newcommand{\Zc}{\mcl{Z}}
\newcommand{\ggt}{\mfr{g}}
\newcommand{\hgt}{\mfr{h}}

\newcommand{\ga}{\gamma}
\newcommand{\Ga}{\Gamma}
\newcommand{\de}{\delta}

\newcommand{\ep}{\varepsilon}
\newcommand{\la}{\lambda}

\newcommand{\ph}{\varphi}
\newcommand{\om}{\omega}
\newcommand{\Om}{\Omega}

\DeclareMathOperator{\Lie}{Lie}
\DeclareMathOperator{\Ker}{Ker}

\DeclareMathOperator{\Ad}{Ad}

\DeclareMathOperator{\rk}{rk}

\DeclareMathOperator{\id}{id}

\DeclareMathOperator{\sta}{st}

\DeclareMathOperator{\ps}{ps}

\newcommand{\GL}{\mbf{GL}}
\newcommand{\SL}{\mbf{SL}}
\newcommand{\Or}{\mbf{O}}
\newcommand{\SO}{\mbf{SO}}

\newcommand{\glg}{\mfr{gl}}

\newcommand{\bom}{\boldmath}
\newcommand{\phm}[1]{\phantom{#1}}

\newcommand{\thra}{\twoheadrightarrow}

\newcommand{\leqn}{\lefteqn}

\newcommand{\olQ}{\leqn{\hskip1pt\ol{\phm{J}}}Q}

\begin{document}

\author{О.\,Г.\?Стырт}
\title{Топологические и~гомологические свойства\\
пространства орбит\\
компактной линейной группы Ли\\
с~коммутативной связной компонентой}
\date{}
\newcommand{\udk}{512.815.1+512.815.6+512.816.1+512.816.2}
\newcommand{\msc}{17B10+17B45+20G05+20G20+22C05+22E45+22E47}

\maketitle

{\leftskip\parind\rightskip\parind
Исследуется вопрос о~том, является ли факторпространство компактной линейной группы топологическим многообразием, а~также является ли оно гомологическим
многообразием. В~данной работе разобран случай бесконечной группы с~коммутативной связной компонентой.


\smallskip

\textbf{Ключевые слова\:} группа Ли, топологический фактор действия.

\smallskip

The problem in question is whether the quotient space of a~compact linear group is a~topological manifold and whether it is a~homological manifold. In
the paper, the case of an infinite group with commutative connected component is researched.

\smallskip

\textbf{Key words\:} Lie group, topological quotient space of an action.\par}

\section{Введение}\label{introd}

Пусть имеется точное линейное представление компактной группы Ли~$G$ в~вещественном векторном пространстве~$V$. Нас будет интересовать вопрос о~том,
является ли фактор $V/G$ этого действия топологическим многообразием, а~также является ли он гомологическим многообразием. Для краткости будем
в~дальнейшем называть топологическое многообразие просто <<многообразием>>.

Пространство~$V$ обладает $G$\д инвариантным скалярным умножением и~поэтому может (и~будет) рассматриваться как евклидово пространство, на котором
группа~$G$ действует ортогональными операторами. Кроме того, поскольку представление $G\cln V$ точное, можно считать, что $G$\т подгруппа Ли группы Ли
$\Or(V)$, а~представление $G\cln V$ тавтологическое.

\begin{df} Линейный оператор в~пространстве над некоторым полем называется \textit{отражением} (соотв. \textit{псевдоотражением}), если подпространство
его неподвижных точек имеет коразмерность $1$ (соотв.~$2$).
\end{df}

К~данному моменту разобран случай конечной группы~$G$. Именно, в~\cite{MAMich} доказывается, что если группа $G\subs\Or(V)$ конечна и~порождена
псевдоотражениями, то $V/G\cong V$. Обратное же утверждение неверно\: из того, что $|G|<\bes$ и~$V/G\cong V$, не следует, что группа $G\subs\Or(V)$
порождена псевдоотражениями. Более исчерпывающий результат получен относительно недавно в~работе~\cite{Lange}. Приведём её основные результаты
(теорема~\ref{lang}), предварительно определив понятие \textit{группы Пуанкаре}.

\begin{df} Рассмотрим компактную группу Ли $S:=\{v\in\Hbb\cln\hn{v}=1\}\subs\Hbb$ (с~операцией умножения кватернионов), накрывающий гомоморфизм
$S\thra\SO_3$ и~прообраз $\Ga\subs S$ группы вращений додекаэдра при указанном гомоморфизме. \textit{Группой Пуанкаре} называется линейная группа,
полученная ограничением действия $S\cln\Hbb$ левыми сдвигами на подгруппу $\Ga\subs S$.
\end{df}

\begin{theorem}\label{lang} Допустим, что группа $G\subs\Or(V)$ конечна.
\begin{nums}{-1}
\item\label{hom} Если $V/G$\т гомологическое многообразие, то имеются разложения $G=G_0\times G_1\sti G_k$ и~$V=V_0\oplus V_1\sop V_k$
\ter{$k\in\Z_{\ge0}$}, такие что
\begin{nums}{-1}
\item подпространства $V_0,V_1\sco V_k\subs V$ попарно ортогональны и~$G$\д инвариантны\~
\item для любых $i,j=0\sco k$ линейная группа $(G_i)|_{V_j}\subs\Or(V_j)$ тривиальна при $i\ne j$, порождена псевдоотражениями при $i=j=0$ и~изоморфна
группе Пуанкаре при $i=j>0$ \ter{в~частности, $\dim V_j=4$ для всякого $j=1\sco k$}.
\end{nums}
\item\label{top} Если разложения из п.~\ref{hom} существуют, а~условие $V/G\cong V$ не выполняется, то $k=1$ и~$V_0=0$.
\end{nums}
\end{theorem}

\begin{proof} См. предложение~3.13 и~теорему~A в~\cite{Lange}.
\end{proof}

\begin{imp}\label{lan} Предположим, что группа $G\subs\Or(V)$ конечна. Пусть $G_{\ps}\nl G$\т подгруппа, порождённая всеми псевдоотражениями группы~$G$.
Если $V/G$\т гомологическое многообразие, то $G_{\ps}\cdot[G,G]=G$.
\end{imp}

\begin{proof} Хорошо известно, что группа Пуанкаре совпадает со своим коммутантом. Осталось применить теорему~\ref{lang}.
\end{proof}

Через~$G^0$ будем обозначать связную компоненту единицы группы~$G$, а~через~$\ggt$\т её касательную алгебру.

В~данной работе рассматривается случай, когда подгруппа $G^0\subs G$ коммутативна\т что равносильно, является тором. Очевидно, что это свойство
сохраняется при переходе к~подгруппе и~к~факторгруппе.

Для произвольного элемента $g\in G$ введём обозначение
\equ{
\om(g):=\rk(E-g)-\rk\br{E-\Ad(g)}\in\Z.}
Положим $\Om:=\bc{g\in G\cln\om(g)\in\{0,2\}}\subs G$ и~$\Om':=\bc{g\in G\cln\om(g)=4,\,\om(g^5)=0}\subs G$.

На пространстве~$\ggt$ определено $\Ad(G)$\д инвариантное скалярное умножение\~ с~помощью последнего мы в~дальнейшем будем отождествлять пространства
$\ggt$ и~$\ggt^*$.

Для произвольного конечного множества~$P$ векторов в~конечномерном пространстве над некоторым полем, рассматриваемого с~учётом кратностей своих элементов,
количество ненулевых векторов множества~$P$ (с~учётом кратностей) будем обозначать через~$\hn{P}$.

Предположим, что подгруппа $G^0\subs G$ коммутативна, т.\,е. является тором.

Любое неприводимое представление группы~$G^0$ одномерно либо двумерно. Напомним введённое в~\cite[\Ss1]{My1} понятие веса её неприводимого представления.

Произвольное двумерное неприводимое представление группы~$G^0$ обладает $G^0$\д инвариантной комплексной структурой, и~мы можем рассматривать его как
одномерное комплексное представление группы~$G^0$, сопоставив ему естественным образом вес\т гомоморфизм групп Ли $\la\cln G^0\to\T$\т и~отождествив
последний с~его дифференциалом\т вектором $\la\in\ggt^*$. Одномерному представлению группы~$G^0$ сопоставим вес $\la:=0\in\ggt^*$.

Классы изоморфных неприводимых представлений группы~$G^0$ характеризуются весами $\la\in\ggt^*=\ggt$, определёнными с~точностью до знака.

Пусть $P\subs\ggt$\т множество весов $\la\in\ggt$, соответствующее разложению представления $G^0\cln V$ в~прямую сумму неприводимых (с~учётом
кратностей). Множество $P\subs\ggt$ не зависит от выбора указанного разложения (с~точностью до знаков весов). Поскольку представление $G\cln V$ точное,
имеем $\ha{P}=\ggt$.

Напомним определения \textit{$q$\д устойчивых} ($q\in\N$) и~\textit{неразложимых} множеств векторов конечномерных пространств над
полями~\cite[\Ss1]{My1}, необходимые и~в~данной работе.

Разложением множества векторов конечномерного линейного пространства на компоненты будем называть его представление в~виде объединения своих подмножеств,
линейные оболочки которых линейно независимы. Если среди указанных линейных оболочек по крайней мере две нетривиальны, то такое разложение назовём
\textit{собственным}. Будем говорить, что множество векторов \textit{неразложимо}, если оно не допускает ни одного собственного разложения на компоненты.
Всякое множество векторов разлагается на неразложимые компоненты единственным образом (с~точностью до распределения нулевого вектора), причём для любого
его разложения на компоненты каждая компонента является объединением некоторых его неразложимых компонент (вновь с~точностью до нулевого вектора).

\begin{df} Конечное множество векторов конечномерного пространства, рассматриваемое с~учётом кратностей своих элементов, назовём
\textit{$q$\д устойчивым} ($q\in\N$), если его линейная оболочка сохраняется при удалении из него любых векторов в~количестве не более~$q$ (с~учётом
кратностей).
\end{df}

\begin{theorem}\label{1st} Если $V/G$\т многообразие, то множество $P\subs\ggt$ является $1$\д устойчивым.
\end{theorem}

\begin{proof} См. предложение~2.2 в~\cite[\Ss2]{My1}.
\end{proof}

В~\cite[\Ss8]{My1} описывается метод сопоставления каждой компактной линейной группе с~коммутативной связной компонентой, $1$\д устойчивым множеством
весов и~фактором~$M$ компактной линейной группы с~коммутативной связной компонентой, $2$\д устойчивым множеством весов и~фактором, гомеоморфным~$M$.
Поэтому в~дальнейшем мы будем рассматривать случай $2$\д устойчивого множества $P\subs\ggt$.

В~п.~\ref{decs} будет доказана следующая теорема.

\begin{theorem}\label{submain} Допустим, что $V/G$\т гомологическое многообразие, а~множество $P\subs\ggt$ является $2$\д устойчивым. Тогда существуют
разложения $G=G_0\times G_1\sti G_p$ и~$V=V_0\oplus V_1\sop V_p$ \ter{$p\in\N$}, такие что
\begin{nums}{-1}
\item подпространства $V_0,V_1\sco V_p\subs V$ попарно ортогональны и~$G$\д инвариантны\~
\item для произвольных $i,j=0\sco p$ линейная группа $(G_i)|_{V_j}\subs\Or(V_j)$ тривиальна при $i\ne j$, конечна при $i=j=0$ и~бесконечна при $i=j>0$\~
\item для всякого $l=0\sco p$ фактор $V_l/G_l$ является гомологическим многообразием\~
\item для произвольного $l=1\sco p$ множество весов представления $G_l\cln V_l$ неразложимо, $2$\д устойчиво и~не содержит нулей.
\end{nums}
\end{theorem}

Если разложения из формулировки теоремы~\ref{submain} существуют, то $V/G$\т гомологическое многообразие\~ если при этом каждый из факторов $V_l/G_l$,
$l=0\sco p$, является многообразием, то $V/G$\т многообразие.

Топологические свойства факторпространства конечной линейной группы описываются теоремой~\ref{lang}, и, таким образом, требуется исследовать случай
представления с~неразложимым $2$\д устойчивым множеством весов, не содержащим нулей, чему будут посвящены теоремы \ref{main} и~\ref{main1}.

Далее до конца введения будем предполагать, что множество $P\subs\ggt$ неразложимо, $2$\д устойчиво и~не содержит нулей. Положим $m:=\dim G\in\N$. Ввиду
соотношения $0\notin P$, пространство~$V$ обладает $G^0$\д инвариантной комплексной структурой.

\begin{theorem}\label{main} Допустим, что $m>1$. Следующие условия \ref{to}---\ref{crit} эквивалентны\:
\begin{nums}{-1}
\item\label{to} $V/G$\т многообразие\~
\item\label{ho} $V/G$\т гомологическое многообразие\~
\item\label{crit} выполняются нижеприведённые условия \ref{pm2}---\ref{finst}\:
\begin{nums}{-1}\renewcommand{\theenumi}{}
\item\label{pm2} $\hn{P}=m+2$\~
\item\label{oplu} пространство~$V$ разлагается в~прямую сумму попарно ортогональных двумерных неприводимых $G^0$\д инвариантных подпространств
$W_1\sco W_{m+2}\subs V$, переставляемых группой~$G$, причём $G(W_1\sop W_m)=W_1\sop W_m$ и, кроме того, $(m>2)\Ra(GW_j=W_j\fa j=1\sco m+2)$\~
\item\label{adj} найдётся элемент $g\in G$, такой что $\Ad(g)=-E$ и~$gW_j=W_j\fa j=1\sco m+2$\~
\item\label{finst} если $v\in V$ и~$|G_v|<\bes$, то $G_v=\ha{G_v\cap\Om}$.
\end{nums}
\end{nums}
\end{theorem}

При $m=1$ будем дополнительно предполагать, что группа $G\subs\Or(V)$ не содержит комплексных отражений,\т к~этому можно свести произвольный случай
(см.~\cite[\Sss3,\,7]{My1}).

\begin{theorem}\label{main1} Допустим, что $m=1$, а~группа $G\subs\Or(V)$ не содержит комплексных отражений. Следующие условия эквивалентны\:
\begin{nums}{-1}
\item\label{to1} $V/G$\т многообразие\~
\item\label{ho1} $V/G$\т гомологическое многообразие\~
\item\label{crit1} $\dim_{\Cbb}V=\hn{P}=3$, $\Ad(G)=\{\pm E\}$, $G=\ha{\Om}$, а~представление $G\cln V$ приводимо.
\end{nums}
\end{theorem}

В~каждой из теорем \ref{main} и~\ref{main1} импликация $\text{\ref{crit}}\Ra\text{\ref{to}}$ доказана (см.~\cite[\Ss1]{My1}, теоремы 1.3 и~1.5
соответственно), а~импликация $\text{\ref{to}}\Ra\text{\ref{ho}}$ очевидна. Осталось доказать импликации $\text{\ref{ho}}\Ra\text{\ref{crit}}$ в~теоремах
\ref{main} и~\ref{main1}, что будет проделано в~п.~\ref{indec}.

\section{Обозначения и~вспомогательные факты}\label{facts}

В~этом параграфе приведён ряд вспомогательных обозначений и~утверждений, в~том числе заимствованных из~\cite{Lange,My1} (все новые утверждения\т
с~доказательствами).

\begin{lemma}\label{prop} Пусть $X$ и~$Y$\т топологические пространства, а~$n$\т натуральное число.
\begin{nums}{-1}
\item Если $X$\т односвязная гомологическая $n$\д сфера, то $X\cong S^n$.
\item Конус над пространством~$X$ является гомологическим $(n+1)$\д многообразием тогда и~только тогда, когда $X$\т гомологическая $n$\д сфера.
\item Пространство $X\times Y$ является гомологическим многообразием тогда и~только тогда, когда $X$ и~$Y$\т гомологические многообразия.
\end{nums}
\end{lemma}

\begin{proof} См. теорему~2.3 и~лемму~2.6 в~\cite[\Ss2]{Lange}.
\end{proof}

\subsection{Элементарные св\'едения линейной алгебры}

\begin{stm} Пусть $g$\т антилинейный оператор в~$n$\д мерном комплексном пространстве~$V$. Тогда $\dim_{\R}V^g\le n$\т что равносильно,
$\rk_{\R}(E-g)\ge n$.
\end{stm}

\begin{proof} Вытекает из очевидных соотношений $V^{-g}=iV^g\subs V$ и~$V^g\cap V^{-g}=0$.
\end{proof}

Напомним основные свойства $q$\д устойчивых ($q\in\N$) конечных множеств векторов конечномерных пространств над полями, которые (множества)
рассматриваются с~учётом кратностей своих элементов~\cite[\Ss1]{My1}.
\begin{nums}{-1}
\item Добавление и~удаление нулевых векторов, а~также умножение векторов на ненулевые элементы поля не влияют на $q$\д устойчивость множества.
\item Образ $q$\д устойчивого множества при линейном отображении пространств является $q$\д устойчивым множеством. В~частности, если некоторое множество
линейных функций на пространстве $q$\д устойчиво, то множество ограничений всех линейных функций данного множества на произвольное подпространство также
$q$\д устойчиво.
\item Для любого разложения произвольного $q$\д устойчивого множества на компоненты (необязательно неразложимые) каждая из компонент является
$q$\д устойчивой.
\item Всякое $q$\д устойчивое множество с~$m$\д мерной линейной оболочкой ($m\in\N$) содержит не менее $m+q$ ненулевых векторов.
\item Всякое $q$\д устойчивое множество с~$m$\д мерной линейной оболочкой ($m\in\N$), содержащее ровно $m+q$ ненулевых векторов, неразложимо, а~любые его
ненулевые векторы в~количестве не более~$m$ линейно независимы.
\end{nums}

Для конечномерного представления конечной группы над произвольным полем следующие условия эквивалентны\:
\begin{nums}{-1}
\item\label{inv} подпространство инвариантов тривиально\~
\item\label{orb} сумма векторов в~любой орбите равна нулю.
\end{nums}
Представление, удовлетворяющее условиям \ref{inv} и~\ref{orb}, будем для краткости называть \textit{представлением без инвариантов} или
\textit{представлением, не имеющим инвариантов}.

Очевидно, что любое подпредставление представления без инвариантов также не имеет инвариантов.

\begin{stm}\label{no1} Никакая группа нечётного порядка не обладает одномерным вещественным представлением без инвариантов.
\end{stm}

\begin{proof} Любое одномерное вещественное представление группы нечётного порядка является тождественным, что немедленно влечёт требуемое.
\end{proof}

\begin{imp} Всякое неприводимое вещественное представление коммутативной группы нечётного порядка, не имеющее инвариантов, двумерно.
\end{imp}

\begin{imp}\label{diev} Если некоторая коммутативная группа имеет нечётный порядок, то размерность любого её вещественного представления без инвариантов
чётна.
\end{imp}

\begin{lemma}\label{impr} Рассмотрим произвольное представление конечной группы~$\Ga$ в~пространстве~$W$ над полем~$\F$, не имеющее инвариантов. Далее,
пусть $W_1\sco W_p\subs W$ \ter{$p\in\N$}\т линейно независимые подпространства, переставляемые группой~$\Ga$, а~$\Ga'$\т подгруппа
$\{\ga\in\Ga\cln\ga W_1=W_1\}\subs\Ga$.
\begin{nums}{-1}
\item Представление $\Ga'\cln W_1$ не имеет инвариантов.
\item Если $\F=\R$, а~$\Ga$\т группа нечётного порядка, то $\dim W_1\ne1$.
\end{nums}
\end{lemma}

\begin{proof} Для всякого вектора $w\in W_1$ имеем $\Ga'w\subs W_1$ и~$(\Ga w)\sm(\Ga'w)\subs W_2\sop W_p$, причём сумма векторов в~орбите $\Ga w\subs W$
равна нулю, вследствие чего сумма векторов подмножества $\Ga'w\subs W_1$ равна нулю. Значит, $\Ga'\cln W_1$\т представление без инвариантов. Если
$\F=\R$, а~$\Ga$\т группа нечётного порядка, то $\Ga'\subs\Ga$\т группа нечётного порядка, и, в~силу утверждения~\ref{no1}, $\dim W_1\ne1$.
\end{proof}

\begin{imp}\label{old} Рассмотрим произвольное вещественное представление группы~$\Ga$ нечётного порядка в~пространстве~$W$, не имеющее инвариантов.
Всякая орбита действия группы~$\Ga$ на множестве прямых пространства~$W$ состоит из нечётного числа линейно зависимых прямых.
\end{imp}

\subsection{Представления компактных групп}

Допустим, что имеется евклидово пространство~$V$, компактная группа Ли~$G$ с~касательной алгеброй~$\ggt$, линейное представление $G\to\Or(V)$ и~его
дифференциал\т представление $\ggt\cln V$. Отображение факторизации $V\to V/G$ будем обозначать через~$\pi$, стабилизатор (соотв. стационарную
подалгебру) вектора $v\in V$\т через~$G_v$ (соотв. через~$\ggt_v$), а~подпространство $(\ggt v)^{\perp}\subs V$ ($v\in V$)\т через~$N_v$.

Пусть $v\in V$\т произвольный вектор. Имеем $\ggt_v=\Lie G_v$, $G_v(\ggt v)=\ggt v$ и~$G_vN_v=N_v$. Положим $M_v:=N_v\cap(N_v^{G_v})^{\perp}\subs N_v$.
Ясно, что $N_v=N_v^{G_v}\oplus M_v\subs V$ и~$G_vM_v=M_v$.

\begin{prop} Если $V/G$\т многообразие, то найдётся связная $G$\д инвариантная окрестность нуля $U\subs V$, для которой фактор $U/G$ гомеоморфен
открытому шару~$B$ размерности $\dim(V/G)$.
\end{prop}

\begin{proof} Точка $\pi(0)\in V/G$ обладает окрестностью, гомеоморфной~$B$. Это означает, что существует $G$\д инвариантная окрестность нуля $U\subs V$,
для которой $U/G\cong B$. Пусть $U^0\subs U$\т компонента линейной связности окрестности~$U$, содержащая точку~$0$. Тогда подмножества $U^0$ и~$U\sm U^0$
окрестности~$U$ открыты и~$G$\д инвариантны. Подмножества $\pi(U^0)\subs\pi(U)$ и~$\pi(U\sm U^0)\subs\pi(U)$ являются открытыми, имеют пустое пересечение
и~дают в~объединении связный фактор $\pi(U)=U/G\cong B$. Следовательно, $U=U^0$ и, таким образом, окрестность $U\subs V$ связна.
\end{proof}

\begin{lemma}\label{locman} Если в~факторе $V/G$ некоторая окрестность точки $\pi(0)$ является \ter{гомологическим} многообразием, то $V/G$\т
\ter{гомологическое} многообразие.
\end{lemma}

\begin{proof} В~пространстве~$V$ найдутся $G$\д инвариантная окрестность нуля~$U$, такая что $U/G$\т (гомологическое) многообразие, и~открытый шар
$B\subs U$ с~центром в~нуле. Имеем $GB=B$, причём $B/G$\т (гомологическое) многообразие. Наконец, существует $G$\д эквивариантный гомеоморфизм $V\to B$,
что влечёт требуемое.
\end{proof}

\begin{theorem}\label{slice} Пусть $v\in V$\т некоторый вектор. Фактор $V/G$ является \ter{гомологическим} многообразием локально в~точке $\pi(v)$ тогда
и~только тогда, когда $N_v/G_v$\т \ter{гомологическое} многообразие.
\end{theorem}

\begin{proof} В~силу теоремы о~слайсе \cite[~гл.~II,~~\S~4\т 5]{Bredon}, фактор $V/G$ локально в~точке $\pi(v)$ гомеоморфен фактору $N_v/G_v$ локально
в~нуле. Осталось применить лемму~\ref{locman}.
\end{proof}

\begin{imp}\label{slim} Пусть $v\in V$\т некоторый вектор. Если $V/G$\т гомологическое многообразие, то $M_v/G_v$\т гомологическое многообразие.
\end{imp}

\begin{proof} Имеем $N_v/G_v\cong N_v^{G_v}\times(M_v/G_v)$. Далее, в~силу теоремы~\ref{slice}, $N_v/G_v$\т гомологическое многообразие.
Осталось применить лемму~\ref{prop}.
\end{proof}

\begin{stm}\label{Mv} В~любом $G^0$\д инвариантном подпространстве $V'\subs V$ существует вектор~$v$, для которого $M_v\subs(V')^{\perp}$.
\end{stm}

\begin{proof} См. утверждение~2.2 в~\cite[\Ss2]{My1}.
\end{proof}

Пусть $v\in V$\т некоторый вектор, такой что $|G_v|<\bes$. Тогда $\ggt_v=0$, и~для любого элемента $g\in G_v$ имеем
\equ{
\dim\br{(E-g)N_v}=\dim\br{(E-g)V}-\dim\br{(E-g)(\ggt v)}=\rk(E-g)-\rk\br{E-\Ad(g)}=\om(g).}
В~частности, элемент $g\in G_v$ принадлежит подмножеству $\Om\subs G$, если и~только если он действует на подпространстве $N_v\subs V$ псевдоотражением
либо тождественно.

\begin{lemma}\label{refst} Если $V/G$\т гомологическое многообразие, то для всякого вектора $v\in V$, такого что $|G_v|<\bes$, имеем
$\ha{G_v\cap\Om}\cdot[G_v,G_v]=G_v$.
\end{lemma}

\begin{proof} Вытекает из теоремы~\ref{slice} и~следствия~\ref{lan}.
\end{proof}

\section{Доказательства результатов}

На протяжении дальнейшей части работы будем предполагать, что $G^0\cong\T^m$, $m\in\N$, и~что представление $G\cln V$ точное. Стабилизатор общего
положения указанного представления конечен, вследствие чего $\dim(V/G)=\dim V-\dim G=\dim V-m$. Множество $P\subs\ggt$ весов представления
$G\cln V$ удовлетворяет равенству $\ha{P}=\ggt$. Легко видеть, что $\Ad(G^0)=\{E\}$ и~$\bgm{\Ad(G)}<\bes$.

Изотипную компоненту представления $G^0\cln V$, соответствующую неприводимым представлениям с~произвольным весом $\la\in P$, обозначим через~$V_{\la}$.
Для всякого $\la\in P\sm\{0\}$ изотипная компонента $V_{\la}\subs V$ обладает структурой комплексного пространства, на котором группа~$G^0$ действует
скалярными линейными операторами. Пространство~$V$ разлагается в~прямую сумму своих попарно ортогональных подпространств $V_{\la}$ ($\la\in P$),
переставляемых группой~$G$.

Имеем $\Ad(G)P=P$ и~$gV_{\la}=V_{\Ad(g)\la}$ ($\la\in P$, $g\in G$). В~частности, $GV_0=V_0$. Кроме того, $V_0=V^{G^0}$. Подпространство
$V_0^{\perp}\subs V$ разлагается в~прямую сумму попарно ортогональных изотипных компонент $V_{\la}\subs V$ ($\la\in P\sm\{0\}$), переставляемых
группой~$G$, и~обладает $G^0$\д инвариантной структурой комплексного пространства размерности~$\hn{P}$.

Пусть $\la\in P\sm\{0\}$\т произвольный вес. Элемент $g\in G$ переводит в~себя изотипную компоненту $V_{\la}\subs V$, если и~только если
$\Ad(g)\la=\pm\la$, действуя на ней при $\Ad(g)\la=\la$ (соотв. при $\Ad(g)\la=-\la$) линейно (соотв. антилинейно) над полем~$\Cbb$.

Пусть $S\subs V$\т единичная сфера евклидова пространства~$V$, $D\subs G$\т подмножество
$\{g\in G\cln V^g\ne0\}=\{g\in G\cln S^g\ne\es\}=\cups{\sst{v\in V\\v\ne0}}G_v=\cups{v\in S}G_v$, а~$G_{\sta}$\т нормальная подгруппа Ли
$\ha{G^0\cup D}\subs G$. Группа Ли $G/G_{\sta}$ конечна, а~действие $(G/G_{\sta})\cln(S/G_{\sta})$ свободное.

\begin{prop}\label{ist} Если $g\in G$ и~$\ggt^{\Ad(g)}\ne0$, то $g\in G_{\sta}$.
\end{prop}

\begin{proof} Рассмотрим подгруппу Ли $H:=\Zc_G(g)\subs G$.

По условию $\hgt:=\Lie H=\ggt^{\Ad(g)}\ne0$, откуда $\dim H>0$. Поскольку представление $H^0\cln V$ точное, некоторая его изотипная компонента
$V'\subs V$ обладает комплексной структурой, такой что $(H^0)|_{V'}=\T E\subs\GL_{\Cbb}(V')$. Далее, $g\in\Zc(H)\subs H$, $\Ad_H(g)=\id_{\hgt}$,
вследствие чего $gV'=V'$ и~$g|_{V'}\in\GL_{\Cbb}(V')$. Найдётся вектор $v\in V'\sm\{0\}$, являющийся собственным для оператора
$g|_{V'}\in\GL_{\Cbb}(V')$. Имеем $gv\in\T v=H^0v$, $g\in H^0G_v\bw\subs G_{\sta}$.
\end{proof}

\begin{imp}\label{kea} Справедливо включение $\Ker\Ad\subs G_{\sta}$.
\end{imp}

\begin{prop}\label{fin} Если $m=1$ и~$V_0=0$, то \ter{априори абстрактная} подгруппа $H:=\ba{(\Ker\Ad)\cap D}\nl G$ конечна.
\end{prop}

\begin{proof} Обозначим через~$n$ число $\hn{P}\in\N$, а~через~$\ph$\т изоморфизм групп Ли $G^0\to\T$.

Пространство~$V$ обладает структурой комплексного пространства с~базисом $\{e_1\sco e_n\}$, удовлетворяющей для некоторых натуральных чисел $k_1\sco k_n$
условиям $G^0\subs\Ker\Ad\subs\GL_{\Cbb}(V)$ и~$ge_j=\br{\ph(g)}^{k_j}\cdot e_j\in V$, где $g\in G^0$ и~$j=1\sco n$. Легко видеть, что
$(\Ker\Ad)\cap D\subs H\subs\Ker\Ad\subs\GL_{\Cbb}(V)$.

Положим $k_0:=|G/G^0|\in\N$, $k:=k_1\dots k_n\in\N$ и~$k':=k_1\spl k_n\in\N$.

Рассмотрим произвольный элемент $h\in(\Ker\Ad)\cap D$. Ясно, что $g:=h^{k_0}\in G^0\cap D$. Найдётся число $j\in\{1\sco n\}$, удовлетворяющее
равенству $\br{\ph(g)}^{k_j}=1$. Заметим, что $g^{k_j}=E$, $g^k=E$, $h^{k_0k}=E$, $(\det_{\Cbb}h)^{k_0k}=\det_{\Cbb}(h^{k_0k})=1$.

Таким образом, равенство $(\det_{\Cbb}h)^{k_0k}=1$ выполняется для любого $h\in(\Ker\Ad)\cap D$, а~значит, и~для любого $h\in H$. В~частности, если
$h\in G^0\cap H$\т произвольный элемент, то $1=(\det_{\Cbb}h)^{k_0k}=\br{\ph(h)}^{k'k_0k}$. Отсюда $|G^0\cap H|=\bgm{\ph(G^0\cap H)}\le k'k_0k<\bes$,
$|H|<\bes$.
\end{proof}

Далее будем считать, что $V/G$\т гомологическое многообразие, а~множество $P\subs\ggt$ является $2$\д устойчивым.

\subsection{Начальные свойства}

Имеем $\hn{P}\ge m+2$, $\dim V\ge2m+4>4$, $\dim S>3$. Поэтому сфера $S\subs V$ связна и~односвязна. Согласно лемме~\ref{prop}, $M:=S/G$\т гомологическая
сфера. Кроме того, $\dim(V/G)=\dim V-m\ge(2m+4)-m>4$, $\dim M>3$. Отсюда (вновь см. лемму~\ref{prop})
\begin{gather}
\br{\pi_1(M)}\bb\bs{\pi_1(M),\pi_1(M)}\cong H_1(M)=0;\label{pi1}\\
\br{\pi_1(M)=\{e\}}\quad\Ra\quad\br{\pi_2(M)=0}\label{pi2}.
\end{gather}

\begin{lemma}\label{comm} Группа $G/G_{\sta}$ совпадает со своим коммутантом.
\end{lemma}

\begin{proof} Поскольку сфера $S\subs V$ связна, факторпространство $S/G_{\sta}$ связно. Далее, как уже отмечалось, $|G/G_{\sta}|<\bes$, а~действие
$(G/G_{\sta})\cln(S/G_{\sta})$ свободное. Фактор данного действия гомеоморфен~$M$, а~отображение факторизации $(S/G_{\sta})\thra M$ является накрытием
со слоем $G/G_{\sta}$. Значит, существует сюръективный гомоморфизм $\pi_1(M)\thra G/G_{\sta}$. Ввиду~\eqref{pi1}, $\pi_1(M)=\bs{\pi_1(M),\pi_1(M)}$, что
влечёт требуемое.
\end{proof}

\begin{theorem}\label{dgen} Если $m=1$ и~$V_0=0$, то $\Ad(D)\ni-E$.
\end{theorem}

\begin{proof} Допустим, что утверждение теоремы не выполняется, т.\,е. что $D\subs\Ker\Ad$.

В~силу предложения~\ref{fin}, $H:=\ha{D}\nl G$\т конечная подгруппа Ли, а~$G':=G/H$\т одномерная группа Ли. Согласно следствию~\ref{kea},
$\Ker\Ad\subs G_{\sta}=\ha{G^0\cup D}\subs\Ker\Ad$, $\Ker\Ad=\ha{G^0\cup D}=G^0H$, $G'/(G')^0\cong G/(G^0H)=G/(\Ker\Ad)\cong\Ad(G)$.

Напомним, что сфера $S\subs V$ связна и~односвязна. Из этого, а~также из соотношений $D=\{g\in G\cln S^g\ne\es\}\subs G$, $H=\ha{D}\nl G$ и~$|H|<\bes$
вытекает, что факторпространство $S/H$ связно и~односвязно, а~действие $G'\cln(S/H)$ свободное. Фактор данного действия гомеоморфен~$M$, отображение
факторизации $(S/H)\thra M$ является локально тривиальным расслоением со слоем~$G'$, и~мы можем рассмотреть участок точной гомотопической
последовательности $\pi_2(M)\to\pi_1(G')\to\pi_1(S/H)\to\pi_1(M)\to G'/(G')^0\to0$. Поскольку $\pi_1(S/H)=\{e\}$ и~$\pi_1(G')\cong\Z$, имеем
$\pi_2(M)\ne0$ и~$\pi_1(M)\cong G'/(G')^0\cong\Ad(G)$. Наконец, ввиду~\eqref{pi1}, $\pi_1(M)=\bs{\pi_1(M),\pi_1(M)}\cong\bs{\Ad(G),\Ad(G)}=\{E\}$,
$\pi_1(M)=\{e\}$. Получили противоречие с~\eqref{pi2}.
\end{proof}

\begin{theorem}\label{wgen} Если $m=1$ и~$V_0=0$, то $\Ad(\Om)\ni-E$.
\end{theorem}

\begin{proof} Допустим, что утверждение теоремы не выполняется, т.\,е. что $\Om\subs\Ker\Ad$.

Имеем $\ha{\Om}\subs\Ker\Ad$ и, кроме того, $[G,G]\subs\Ker\Ad$, откуда $\ha{\Om}\cdot[G,G]\subs\Ker\Ad$.

Рассмотрим произвольный вектор $v\in S$. Как легко заметить, $|G_v|<\bes$. Согласно лемме~\ref{refst},
$G_v=\ha{G_v\cap\Om}\cdot[G_v,G_v]\subs\ha{\Om}\cdot[G,G]\subs\Ker\Ad$.

Мы видим, что для всякого $v\in S$ выполнено включение $G_v\subs\Ker\Ad$. Следовательно, $D=\cups{v\in S}G_v\subs\Ker\Ad$, что невозможно в~силу
теоремы~\ref{dgen}.
\end{proof}

\begin{theorem}\label{dim3} Если $m=1$ и~$V_0=0$, то $\Ad(G)=\{\pm E\}$ и~$\hn{P}=3$.
\end{theorem}

\begin{proof} Согласно теореме~\ref{wgen}, найдётся элемент $g\in\Om$, такой что $\Ad(g)=-E$. Имеем $\Ad(G)=\{\pm E\}$, $n:=\hn{P}\ge3$,
$\rk\br{E-\Ad(g)}=1$ и~$\om(g)\le2$, откуда $\rk(E-g)\le3$. Пространство~$V$ обладает структурой $n$\д мерного комплексного пространства, на котором
элемент $g\in G$ действует антилинейно. Значит, $\rk(E-g)\ge n\ge3\ge\rk(E-g)$, $n=3$.
\end{proof}

Для конечного подмножества $Q\subs\ggt^*$, рассматриваемого с~учётом кратностей своих элементов, подалгебры $\hgt\subs\ggt$ и~вектора
$\xi\in\ggt$ положим $Q|_{\hgt}:=\{\la|_{\hgt}\in\hgt^*\cln\la\in Q\}\subs\hgt^*$ и~$Q_{\xi}:=\{\la\in Q\cln\la(\xi)\ne0\}\subs Q$.

Пусть $v\in V$\т произвольный вектор. Имеем $\ggt_v(\ggt v)=\ggt(\ggt_v v)=0$, $\ggt v\subs V^{G_v^0}\subs V$ и, как следствие,
$M_v^{\perp}=(\ggt v)\oplus N_v^{G_v}\subs V^{G_v^0}\subs V$. Представление $G_v\cln V$ точное, а~множество его весов совпадает с~множеством
$P|_{\ggt_v}\subs\ggt_v^*$. Значит, представление $G_v^0\cln M_v$ точное, а~множество его весов с~точностью до нулей совпадает с~множеством
$P|_{\ggt_v}\subs\ggt_v^*$.

\begin{theorem}\label{di3} Если $m=1$, то $\Ad(G)=\{\pm E\}$ и~$\hn{P}=3$.
\end{theorem}

\begin{proof} Согласно утверждению~\ref{Mv}, существует вектор $v\in V_0$, для которого $M_v\subs V_0^{\perp}$. Далее, в~силу следствия~\ref{slim},
$M_v/G_v$\т гомологическое многообразие. Имеем $G_v\sups G^0$, $G_v^0=G^0$, $\ggt_v=\Lie G_v=\ggt$, $\dim G_v=1$. Представление $G_v^0\cln M_v$ точное,
а~множество его весов с~точностью до нулей совпадает с~множеством $P\subs\ggt$. При этом $M_v\subs V_0^{\perp}$, откуда $M_v^{G_v^0}=M_v^{G^0}=0$.
Осталось применить теорему~\ref{dim3} к~представлению $G_v\cln M_v$.
\end{proof}

\begin{lemma}\label{hp3} Пусть $Q\subs P$\т подмножество, для которого $\caps{\la\in Q}(\Ker\la)=\R\xi\subs\ggt$, $\xi\in\ggt\sm\{0\}$. Тогда
$\Ad(G)\xi\ni-\xi$ и~$\hn{P_{\xi}}=3$.
\end{lemma}

\begin{proof} В~каждой изотипной компоненте $V_{\la}\subs V$ ($\la\in Q$) выберем ненулевой вектор~$v_{\la}$. Положим $v:=\sums{\la\in Q}v_{\la}\in V$.
Имеем $\Lie G_v=\ggt_v=\caps{\la\in Q}(\Ker\la)=\R\xi\subs\ggt$, $\dim G_v=1$.

Представление $G_v^0\cln M_v$ точное, а~множество его весов с~точностью до нулей совпадает с~$2$\д устойчивым множеством $P|_{\ggt_v}\subs\ggt_v^*$.
Согласно следствию~\ref{slim}, $M_v/G_v$\т гомологическое многообразие. Далее, в~силу теоремы~\ref{di3},
$\br{\Ad(G_v)}|_{\ggt_v}=\{\pm\id_{\ggt_v}\}\subs\Or(\ggt_v)$, $\Ad(G_v)\xi=\{\pm\xi\}\subs\ggt_v$, $\Ad(G)\xi\ni-\xi$, а~также
$\bn{P|_{\ggt_v}}=3$, $\hn{P_{\xi}}=\bn{P|_{(\R\xi)}}=3$.
\end{proof}

\begin{lemma}\label{exa} Всякая неразложимая компонента~$Q$ множества $P\subs\ggt$ удовлетворяет равенству $\hn{Q}=\dim\ha{Q}+2$.
\end{lemma}

\begin{proof} Поскольку $\ha{P}=\ggt$, множество $P\subs\ggt=\ggt^*$ представляется в~виде объединения (с~учётом кратностей векторов) своих подмножеств
$P'$ и~$P\"$, таких что подмножество $P'\subs\ggt^*$ не содержит кратных векторов и~совпадает с~некоторым базисом $\{\la_1\sco\la_m\}$
пространства~$\ggt^*$. Существует базис $\{\xi_1\sco\xi_m\}$ пространства~$\ggt$, удовлетворяющий равенствам $\la_i(\xi_j)=\de_{ij}$ ($i,j=1\sco m$).

Имеем $\hnn{P'_{\xi_j}}=1$ ($j=1\sco m$) и~$\hnn{P'_{\xi_{j_1}}\swo P'_{\xi_{j_2}}}=2$ ($j_1,j_2=1\sco m$, $j_1\ne j_2$). Далее, для всякого $j=1\sco m$
пересечение ядер всех весов $\la_i\in P$ ($i=1\sco m$, $i\ne j$) есть не что иное как подпространство $\R\xi_j\subs\ggt$, $\xi_j\in\ggt\sm\{0\}$, и,
согласно лемме~\ref{hp3}, $\hnn{P_{\xi_j}}=3$, $\hnn{P\"_{\xi_j}}=\hnn{P_{\xi_j}}-\hnn{P'_{\xi_j}}=2$.

Как легко заметить, $P\"\sm\{0\}=\cupl{j=1}{m}P\"_{\xi_j}\subs P\"$.

Покажем, что среди подмножеств $P\"_{\xi_j}\subs P\"$ ($j=1\sco m$) никакие два различных не пересекаются.

Допустим, что  $P\"_{\xi_{j_1}}\ne P\"_{\xi_{j_2}}$ и~$P\"_{\xi_{j_1}}\cap P\"_{\xi_{j_2}}\ne\es$ для некоторых $j_1,j_2\in\{1\sco m\}$.

Ясно, что $j_1\ne j_2$. Значит, $\hnn{P'_{\xi_{j_1}}\swo P'_{\xi_{j_2}}}=2$. Кроме того, найдётся вес $\la\in P$, такой что $c_1:=\la(\xi_{j_1})\ne0$
и~$c_2:=\la(\xi_{j_2})\ne0$. Пересечение ядер всех весов $\la_i\in P$ ($i=1\sco m$, $i\ne j_1,j_2$) и~$\la\in P$ совпадает с~подпространством
$\R\xi\subs\ggt$, $\xi:=c_2\xi_{j_1}-c_1\xi_{j_2}\in\ggt\sm\{0\}$, и, в~силу леммы~\ref{hp3}, $\hn{P_{\xi}}=3$. Из соотношения $c_1,c_2\ne0$ вытекает,
что $P_{\xi_{j_1}}\swo P_{\xi_{j_2}}\subs P_{\xi}$, $\hnn{P_{\xi_{j_1}}\swo P_{\xi_{j_2}}}\le\hn{P_{\xi}}=3$,
$\hnn{P\"_{\xi_{j_1}}\swo P\"_{\xi_{j_2}}}=\hnn{P_{\xi_{j_1}}\swo P_{\xi_{j_2}}}-\hnn{P'_{\xi_{j_1}}\swo P'_{\xi_{j_2}}}\le1$. В~то же время
$\hnn{P\"_{\xi_{j_1}}}=\hnn{P\"_{\xi_{j_2}}}=2$ и~$P\"_{\xi_{j_1}}\ne P\"_{\xi_{j_2}}$, откуда
$\hnn{P\"_{\xi_{j_1}}\swo P\"_{\xi_{j_2}}}\ge2$. Получили противоречие.

Тем самым мы установили, что среди подмножеств $P\"_{\xi_j}\subs P\"$ ($j=1\sco m$) никакие два различных не пересекаются.

Следовательно, существуют разложения $\{1\sco m\}=\sqcupl{l=1}{p}I_l$ и~$P\"\sm\{0\}=\sqcupl{l=1}{p}Q\"_l\subs P\"$ ($p\in\N$), где
$I_l\subs\{1\sco m\}$, $I_l\ne\es$, $Q\"_l\subs P\"\sm\{0\}$, $\hn{Q\"_l}=2$ ($l=1\sco p$) и, кроме того, $P\"_{\xi_j}=Q\"_l\subs P\"$ ($l=1\sco p$,
$j\in I_l$).

Рассмотрим произвольное число $l\in\{1\sco p\}$.

Пусть $Q'_l\subs P'$\т подмножество, включающее в~себя каждый из векторов $\la_i\in\ggt^*$, $i\in I_l$, с~кратностью~$1$ и~не содержащее других векторов
пространства~$\ggt^*$, а~$Q_l\subs P$\т объединение (с~учётом кратностей векторов) подмножеств $Q'_l\subs P'$ и~$Q\"_l\subs P\"$. Если
$\la\in Q\"_l\subs P\"$ и~$j\in\{1\sco m\}\sm I_l$, то $\la\notin P\"_{\xi_j}$, $\la(\xi_j)=0$. Значит,
$Q\"_l\subs\ha{\la_i}_{i\in I_l}=\ha{Q'_l}\subs\ggt^*$, $\ha{Q_l}=\ha{Q'_l}\subs\ggt^*$, $\dim\ha{Q_l}=\dim\ha{Q'_l}=|Q'_l|=\hn{Q'_l}$,
$\hn{Q_l}=\hn{Q'_l}+\hn{Q\"_l}=\dim\ha{Q_l}+2$.

Имеем $\ggt^*=\oplusl{l=1}{p}\ha{Q'_l}=\oplusl{l=1}{p}\ha{Q_l}$ и~$P'=\sqcupl{l=1}{p}Q'_l\subs P\sm\{0\}$, откуда $P\sm\{0\}=\sqcupl{l=1}{p}Q_l\subs P$.
Таким образом, множество $P\sm\{0\}\subs P$ разлагается на компоненты $Q_l\subs P$ ($l=1\sco p$). Для всякого $l=1\sco p$ компонента $Q_l\subs P$
является $2$\д устойчивой и, в~силу равенства $\hn{Q_l}=\dim\ha{Q_l}+2$, неразложимой. Это завершает доказательство.
\end{proof}

\subsection{Порождающие стабилизаторы}

Нашей ближайшей целью является доказательство следующей теоремы.

\begin{theorem}\label{gst} Имеем $G_{\sta}=G$.
\end{theorem}

Доказательству теоремы~\ref{gst} предпошлём несколько вспомогательных утверждений.

\begin{lemma}\label{cod} Пусть $\Ga\subs\Ad(G)$\т коммутативная подгруппа нечётного порядка. Тогда $\ggt^{\Ga}\ne0$.
\end{lemma}

\begin{proof} Предположим, что $\ggt^{\Ga}=0$, т.\,е. что тавтологическое представление $\Ga\cln\ggt$ не имеет инвариантов.

Пусть $Q\subs\ggt$\т некоторая неразложимая компонента множества $P\subs\ggt$, $\Ga'$\т подгруппа
$\{\ga\in\Ga\cln\ga Q=Q\}=\bc{\ga\in\Ga\cln\ga\ha{Q}=\ha{Q}}\subs\Ga$, а~$\olQ$\т множество всевозможных прямых $\R\la\subs\ggt$, $\la\in Q\sm\{0\}$
(\textit{без учёта кратностей прямых}). Ясно, что $\Ga'\olQ=\olQ$.

Линейные оболочки всех неразложимых компонент множества $P\subs\ggt$ суть линейно независимые подпространства пространства~$\ggt$, переставляемые
группой~$\Ga$, и, согласно лемме~\ref{impr}, представление $\Ga'\cln\ha{Q}$ не имеет инвариантов. При этом $\Ga'\subs\Ga$\т коммутативная группа
нечётного порядка. В~силу следствий \ref{diev} и~\ref{old}, число $m':=\dim\ha{Q}\in\N$ чётно, а~любая орбита действия $\Ga'\cln\olQ$ состоит из
нечётного числа линейно зависимых прямых. В~частности, $m'\ge2$.

Неразложимая компонента $Q\subs\ggt$ множества $P\subs\ggt$ является $2$\д устойчивой и, согласно лемме~\ref{exa}, удовлетворяет равенству
$\hn{Q}=\dim\ha{Q}+2$. Значит, в~множестве $Q\subs\ggt$ любые ненулевые векторы в~количестве не более~$m'$ линейно независимы. Поскольку $m'\ge2$,
все ненулевые векторы данного множества попарно не пропорциональны. Следовательно, множество~$\olQ$ включает в~себя ровно $m'+2$ прямых, среди которых
любые в~количестве не более~$m'$ линейно независимы.

Ввиду вышесказанного, всякая орбита действия $\Ga'\cln\olQ$ имеет нечётный порядок, не меньший $m'+1$. При этом $|\olQ|=m'+2\dv2$, и, значит, порядок
любой орбиты действия $\Ga'\cln\olQ$ равен $m'+1$. Отсюда $m'+2=|\olQ|\dv m'+1$, $1\dv m'+1>1$. Получили противоречие.
\end{proof}

\begin{imp}\label{odr} Если $A\in\Ad(G)$ и~$A^k=E$, где $k$\т нечётное натуральное число, то $\ggt^A\ne0$.
\end{imp}

\begin{proof} Подгруппа $\Ga:=\ha{A}\subs\Ad(G)$ является коммутативной и~имеет нечётный порядок. В~силу леммы~\ref{cod}, $\ggt^A=\ggt^{\Ga}\ne0$.
\end{proof}

Число $\bgm{\Ad(G)}\in\N$ представимо в~виде $2^{d_0}k_0$, где $d_0\in\Z_{\ge0}$, а~$k_0$\т нечётное натуральное число.

\begin{prop} Пусть $g\in G$\т произвольный элемент. Тогда $h:=g^{2^{d_0}}\in G_{\sta}$.
\end{prop}

\begin{proof} Положим $A:=\Ad(h)=\br{\Ad(g)}^{2^{d_0}}\in\Ad(G)$. Имеем $A^{k_0}=\br{\Ad(g)}^{2^{d_0}k_0}=E$ и, в~силу следствия~\ref{odr},
$\ggt^{\Ad(h)}=\ggt^A\ne0$. Согласно предложению~\ref{ist}, $h\in G_{\sta}$.
\end{proof}

\begin{imp}\label{po2} Для всякого $\ga\in G/G_{\sta}$ имеем $\ga^{2^{d_0}}=e$.
\end{imp}

Теперь мы можем доказать теорему~\ref{gst}.

Напомним, что группа Ли $G/G_{\sta}$ конечна. В~силу следствия~\ref{po2}, указанная группа является конечной $2$\д группой и~потому разрешима. С~другой
стороны, согласно лемме~\ref{comm}, группа $G/G_{\sta}$ совпадает со своим коммутантом. Отсюда $G/G_{\sta}=\{e\}$, $G_{\sta}=G$.

Тем самым теорема~\ref{gst} доказана.

\begin{theorem}\label{fst} Группа~$G$ порождена объединением своих подгрупп $G^0$ и~$G_v$ \ter{$v\in V$, $|G_v|<\bes$}.
\end{theorem}

\begin{proof} Предположим, что утверждение теоремы не выполняется.

Докажем, что существует точное представление $G'\cln V'$ компактной группы Ли~$G'$, которое имеет размерность менее $\dim V$ и, аналогично представлению
$G\cln V$, удовлетворяет следующим условиям\:
\begin{nums}{-1}
\item\label{tori} $(G')^0\cong\T^{m'}$, $m'\in\N$\~
\item\label{2st} множество весов представления $G'\cln V'$ является $2$\д устойчивым\~
\item\label{homa} $V'/G'$\т гомологическое многообразие\~
\item\label{noge} группа~$G'$ не порождается своими подгруппами $(G')^0$ и~$G'_{v'}$ ($v'\in V'$, $|G'_{v'}|<\bes$).
\end{nums}

Согласно теореме~\ref{gst}, $\ha{G^0\cup D}=G_{\sta}=G$. Другими словами, группа~$G$ порождается своими подгруппами $G^0$ и~$G_v$
($v\in V\sm\{0\}$). В~то же время в~группе~$G$ объединение подгрупп $G^0$ и~$G_v$ ($v\in V$, $|G_v|<\bes$) порождает подгруппу $H\ne G$.
Значит, существует вектор $v\in V\sm\{0\}$, такой что подгруппа $G_v\subs G$ является бесконечной и~не содержится в~подгруппе $H\subs G$.

Представлению $G_v\cln M_v$ отвечает гомоморфизм групп Ли $R\cln G_v\to\Or(M_v)$. Положим $V':=M_v\subs N_v\subs V$ и~$G':=R(G_v)\subs\Or(M_v)=\Or(V')$.

Докажем, что тавтологическое представление $G'\cln V'$ является искомым.

Очевидно, что представление $G'\cln V'$ точное.

Имеем $G_v^0\cong\T^{m'}$, где $m':=\dim G_v\in\N$. Кроме того, представление $G_v^0\cln M_v$ точное, и~поэтому $(\Ker R)\cap G_v^0=\{e\}\subs G_v^0$.
Отсюда $(G')^0=R(G_v^0)\cong G_v^0\cong\T^{m'}$ и~$|\Ker R|<\bes$.

Множество весов представления $G_v\cln M_v$ с~точностью до нулей совпадает с~$2$\д устойчивым множеством $P|_{\ggt_v}\subs\ggt_v^*$. Поскольку
$|\Ker R|<\bes$, множество весов представления $G'\cln V'$ также $2$\д устойчиво.

Ясно, что $M_v/G_v\cong V'/G'$. В~силу следствия~\ref{slim}, фактор $M_v/G_v$ является гомологическим многообразием\~ то же можно сказать и~о~факторе
$V'/G'$.

Заметим, что $0\ne v\in N_v^{G_v}\subs(\ggt v)\oplus N_v^{G_v}=M_v^{\perp}=(V')^{\perp}\subs V$. Значит, $\dim V'<\dim V$.

Пусть $U$\т подмножество $\bc{v'\in V'\cln|G'_{v'}|<\bes}\subs V'$, а~$H'$\т подгруппа группы~$G'$, порождённая подгруппами $(G')^0\subs G'$
и~$G'_{v'}\subs G'$ ($v'\in U$). Покажем, что $H'\ne G'$.

В~силу соотношения $(G')^0\cong\T^{m'}$, стабилизатор общего положения точного представления $G'\cln V'$ конечен. Следовательно, $U\ne\es$.

Рассмотрим произвольный вектор $v'\in U$. Имеем $|G'_{v'}|<\bes$ и~$|\Ker R|<\bes$. Поэтому $\bgm{R^{-1}(G'_{v'})}<\bes$. Кроме того,
$v'\in U\subs V'=M_v\subs N_v$. Значит, найдётся число $\ep\in\R_{>0}$, для которого $G_{v+\ep v'}\subs G_v$. Очевидно, что
$G_{v+\ep v'}=G_v\cap G_{v'}=R^{-1}(G'_{v'})\subs G_v$. Таким образом, $R^{-1}(G'_{v'})=G_{v+\ep v'}\subs G_v$ и~$\bgm{R^{-1}(G'_{v'})}<\bes$,
откуда $\Ker R\subs R^{-1}(G'_{v'})\subs G_v\cap H$, $G'_{v'}=R\br{R^{-1}(G'_{v'})}\subs R(G_v\cap H)$.

Тем самым стало ясно, что для любого $v'\in U$ выполнены включения $\Ker R\subs G_v\cap H$ и~$G'_{v'}\subs R(G_v\cap H)$. Поскольку $U\ne\es$, имеем
$\Ker R\subs G_v\cap H$. Заметим, что $G^0\subs H$, $G_v^0\subs G_v\cap G^0\subs G_v\cap H$, $(G')^0=R(G_v^0)\subs R(G_v\cap H)$. Следовательно,
$H'\subs R(G_v\cap H)$, $R^{-1}(H')\subs R^{-1}\br{R(G_v\cap H)}=(\Ker R)(G_v\cap H)\subs(G_v\cap H)(G_v\cap H)=G_v\cap H\subs H$.

Если $H'=G'$, то $G_v=R^{-1}(H')\subs H$, что неверно. Поэтому $H'\ne G'$.

Мы видим, что представление $G'\cln V'$ является точным, а~также удовлетворяет неравенству $\dim V'<\dim V$ и~условиям \ref{tori}---\ref{noge}.

Рассуждая аналогично, мы получим бесконечную последовательность точных представлений $G'\cln V'$ компактных групп Ли~$G'$, в~которой
\begin{itemize}
\item каждое из представлений $G'\cln V'$ удовлетворяет условиям \ref{tori}---\ref{noge}\~
\item размерности пространств~$V'$ строго убывают.
\end{itemize}
Это приводит нас к~противоречию, завершающему доказательство.
\end{proof}

\subsection{Специальные элементы группы}

Обозначим через~$\wt{V}$ комплексное пространство $V\otimes\Cbb$, а~через~$\tau$\т оператор комплексной структуры
$\wt{V}\to\wt{V},\,x+yi\to x-yi,\,x,y\in V$. Представление $G\cln V$ естественным образом индуцирует комплексное представление $G\cln\wt{V}$. Пусть
$\wt{V}_1\sco\wt{V}_n\subs\wt{V}$ ($n\in\N$)\т изотипные компоненты комплексного представления $G^0\cln\wt{V}$. Имеем $\wt{V}=\oplusl{j=1}{n}\wt{V}_j$.
Любое неприводимое комплексное представление группы~$G^0$ одномерно, что влечёт соотношения $(G^0)|_{\wt{V}_j}\subs\T E\subs\GL_{\Cbb}(\wt{V}_j)$,
$j=1\sco n$. Группа~$G$ переставляет подпространства $\wt{V}_1\sco\wt{V}_n\subs\wt{V}$, причём
\eqn{\label{kad}
\Ker\Ad=\{g\in G\cln g\wt{V}_j=\wt{V}_j\,\fa j=1\sco n\}\subs G.}
Оператор $\tau\cln\wt{V}\to\wt{V}$ также переставляет подпространства $\wt{V}_1\sco\wt{V}_n$ пространства~$\wt{V}$. Если $j\in\{1\sco n\}$
и~$\tau\wt{V}_j=\wt{V}_j$, то каждый оператор подалгебры $\ggt|_{\wt{V}_j}\subs i\R E\subs\glg_{\Cbb}(\wt{V}_j)$ антикоммутирует с~оператором
$\tau|_{\wt{V}_j}\cln\wt{V}_j\to\wt{V}_j$, и, поскольку $[\ggt,\tau]=0$, выполнено равенство $\ggt\wt{V}_j=\ggt\tau\wt{V}_j=0$.

Рассмотрим произвольный элемент $g\in G$.

Положим $A:=\Ad(g)\in\Or(\ggt)$, $V^0:=\BC{v\in V\cln\br{(E-A)\ggt}v=0}=\opluss{\la\in P^A}V_{\la}\subs V$ и, кроме того,
$\wt{V}^0:=\BC{v\in\wt{V}\cln\br{(E-A)\ggt}v=0}=V^0\oplus iV^0\subs\wt{V}$.

Элемент $g\in G$ переставляет подпространства $\wt{V}_1\sco\wt{V}_n\subs\wt{V}$. Значит, существуют числа $p\in\N$, $k_1\sco k_p\in\N$,
$n_1\sco n_p\in\{1\sco n\}$ и~подпространства $\wt{V}^1\sco\wt{V}^p\subs\wt{V}$, для которых $\wt{V}^l=\oplusl{j=0}{k_l-1}g^j\wt{V}_{n_l}$,
$g^{k_l}\wt{V}_{n_l}=\wt{V}_{n_l}$ ($l=1\sco p$) и~$\wt{V}=\oplusl{l=1}{p}\wt{V}^l$. Без ограничения общности можно считать, что
$k_1\sco k_{p'}\ge2$ и~$k_{p'+1}\seq k_p=1$, где $p'\in\{0\sco p\}$.

Как легко заметить, $\wt{V}^0=\oplusl{l=p'+1}{p}\wt{V}^l\subs\wt{V}$, $\wt{V}=\oplusl{l=0}{p'}\wt{V}^l$. При этом $g\wt{V}^l=\wt{V}^l$ для всякого
$l=0\sco p$. Отсюда $\rk(E-g)=\dim\br{(E-g)V}=\dim_{\Cbb}\br{(E-g)\wt{V}}=\suml{l=0}{p'}\dim_{\Cbb}\br{(E-g)\wt{V}^l}$.

Пусть $l\in\{1\sco p'\}$\т произвольное число. Положим $d_l:=\dim_{\Cbb}\wt{V}_{n_l}\in\N$. Ясно, что $\dim_{\Cbb}\wt{V}^l=k_ld_l$
и~$\dim_{\Cbb}(\wt{V}^l)^g\le d_l$. Значит, $\dim_{\Cbb}\br{(E-g)\wt{V}^l}\ge(k_l-1)d_l$.

Ввиду вышесказанного, $\rk(E-g)\ge\dim_{\Cbb}\br{(E-g)\wt{V}^0}+\suml{l=1}{p'}(k_l-1)d_l$. Кроме того,
$2\cdot\bn{(E-A)P}=2\cdot\hn{P\sm P^A}=\dim(V/V^0)=\dim_{\Cbb}(\wt{V}/\wt{V}^0)=\suml{l=1}{p'}\dim_{\Cbb}\wt{V}^l=\suml{l=1}{p'}k_ld_l$, откуда
$\rk(E-g)-\bn{(E-A)P}\ge\dim_{\Cbb}\br{(E-g)\wt{V}^0}+\frac{1}{2}\cdot\suml{l=1}{p'}(k_l-2)d_l\ge\dim\br{(E-g)V^0}+\frac{1}{2}\cdot\suml{l=1}{p'}(k_l-2)$.

\begin{stm}\label{nor} Справедливо неравенство $\rk(E-g)\ge\bn{(E-A)P}$. При этом равенство $\rk(E-g)=\bn{(E-A)P}$ возможно лишь в~случае
$\opluss{\la\in P^A}V_{\la}\subs V^g$ и~$A^2=E$.
\end{stm}

\begin{proof} Имеем $\rk(E-g)-\bn{(E-A)P}\ge\dim\br{(E-g)V^0}+\frac{1}{2}\cdot\suml{l=1}{p'}(k_l-2)\ge0$. Допустим, что $\rk(E-g)-\bn{(E-A)P}=0$. Тогда
$(E-g)V^0=0$ и~$k_1\seq k_{p'}=2$. Следовательно, $\opluss{\la\in P^A}V_{\la}=V^0\subs V^g$. Кроме того, $k_1\sco k_p\in\{1,2\}$, и, значит,
$g^2\wt{V}_j=\wt{V}_j$, $j=1\sco n$. В~силу~\eqref{kad}, $\Ad(g^2)=E$.
\end{proof}

\begin{stm}\label{norp} Если $A^5=E$ и~$\rk(E-g)-\bn{(E-A)P}\le2$, то $A=E$.
\end{stm}

\begin{proof} По условию $\Ad(g^5)=E$. В~силу~\eqref{kad}, $g^5\wt{V}_j=\wt{V}_j$ для любого $j=1\sco n$. Отсюда $k_1\sco k_p\in\{1,5\}$,
$k_1\seq k_{p'}=5$, $2\ge\rk(E-g)-\bn{(E-A)P}\ge\frac{1}{2}\cdot\suml{l=1}{p'}(k_l-2)=\frac{3p'}{2}$, $p'\le1$.

Напомним, что оператор $\tau\cln\wt{V}\to\wt{V}$ переставляет подпространства $\wt{V}_1\sco\wt{V}_n\subs\wt{V}$, причём если $j\in\{1\sco n\}$
и~$\tau\wt{V}_j=\wt{V}_j$, то $\ggt\wt{V}_j=0$, $\wt{V}_j\subs\wt{V}^0$. Далее, $\wt{V}^0=\oplusl{l=p'+1}{p}\wt{V}^l\subs\wt{V}$
и~$\tau\wt{V}^0=\wt{V}^0$. Значит, $5p'=k_1\spl k_{p'}=\Bm{\bc{j\in\{1\sco n\}\cln\wt{V}_j\cap\wt{V}^0=0}}\dv2$, $p'\dv2$.

Поскольку $p'\le1$ и~$p'\dv2$, имеем $p'=0$, $k_1\seq k_p=1$. Таким образом, $g\wt{V}_j=\wt{V}_j$, $j=1\sco n$. В~силу~\eqref{kad}, $\Ad(g)=E$.
\end{proof}

Множество $(E-A)P\subs\ggt$ является $2$\д устойчивым. Его линейная оболочка есть не что иное как подпространство $(E-A)\ggt\subs\ggt$ размерности
$r:=\rk(E-A)\in\Z_{\ge0}$. Значит,
\begin{itemize}
\item если $A\ne E$, то $\bn{(E-A)P}-r\ge2$\~
\item если $A\ne E$ и~$\bn{(E-A)P}-r=2$, то множество $(E-A)P\subs\ggt$ неразложимо, а~любые его ненулевые векторы в~количестве не более~$r$ линейно
независимы.
\end{itemize}

\begin{stm}\label{ane} Допустим, что $A\ne E$. Тогда $\om(g)\ge2$, причём если $\om(g)=2$, то $\opluss{\la\in P^A}V_{\la}\subs V^g$, $A^2=E$
и~$\bn{(E-A)P}-r=2$.
\end{stm}

\begin{proof} Имеем $\om(g)=\Br{\rk(E-g)-\bn{(E-A)P}}+\Br{\bn{(E-A)P}-r}$. Осталось применить утверждение~\ref{nor}.
\end{proof}

\begin{imp}\label{aneo} Предположим, что $g\in\Om$ и~$A\ne E$. Тогда $\opluss{\la\in P^A}V_{\la}\subs V^g$, $A^2=E$, $\bn{(E-A)P}-r=2$, множество
$(E-A)P\subs\ggt$ неразложимо, а~любые его ненулевые векторы в~количестве не более~$r$ линейно независимы.
\end{imp}

\begin{lemma}\label{omp} Если $g\in\Om'$, то $A=E$.
\end{lemma}

\begin{proof} Допустим, что $A\ne E$.

Имеем $\bn{(E-A)P}-r\ge2$. Далее, согласно условию, $\om(g)=4$ и~$\om(g^5)=0$. Значит, $\rk(E-g)-\bn{(E-A)P}=\om(g)-\Br{\bn{(E-A)P}-r}\le2$. В~силу
утверждения~\ref{norp}, $A^5\ne E$, $\Ad(g^5)\ne E$. Применяя \textit{к~элементу $g^5\in G$} утверждение~\ref{ane}, получаем, что $\om(g^5)\ge2$.
Это противоречит равенству $\om(g^5)=0$.
\end{proof}

\begin{lemma}\label{apme} Допустим, что $g\in\Om$, $A\ne E$ и~$P=P^A\cup P^{-A}$. Тогда среди неразложимых компонент множества $P\subs\ggt$ одна
содержится в~подпространстве $\ggt^{-A}\subs\ggt$, а~все остальные\т в~подпространстве $\ggt^A\subs\ggt$.
\end{lemma}

\begin{proof} Поскольку $\ggt^A\cap\ggt^{-A}=0$, множество $P\subs\ggt$ разлагается на компоненты $P^A\subs P$ и~$P^{-A}\subs P$. Согласно
следствию~\ref{aneo}, множество $(E-A)P\subs\ggt$ неразложимо. При этом множество $(E-A)P\subs\ggt$ с~точностью до нулей совпадает с~множеством
$2P^{-A}\subs\ggt$.
\end{proof}

\begin{lemma}\label{are} Предположим, что $g\in\Om$ и~$P\ne P^A\cup P^{-A}$. Тогда $r=1$, а~множество $P\sm P^A\subs\ggt$ неразложимо. Кроме того,
$\hn{P\sm P^A}=3$, $\dim\ha{P\sm P^A}=2$ и~$\hn{P^{-A}}=1$.
\end{lemma}

\begin{proof} По условию найдутся векторы $\la_1,\la_2\in P$, такие что $A\la_1=\la_2$ и~$\dim\ha{\la_1,\la_2}=2$\~ в~частности, $A\ne E$, $r>0$.
В~силу следствия~\ref{aneo}, $A^2=E$, $\hn{P\sm P^A}=\bn{(E-A)P}=r+2$, а~любые ненулевые векторы множества $(E-A)P\subs\ggt$ в~количестве не более~$r$
линейно независимы. Имеем $A\la_2=A^2\la_1=\la_1$. Далее, $\la_1,\la_2\in P$, $\la_2\ne\pm\la_1$, а~ненулевые векторы $(E-A)\la_1=\la_1-\la_2$
и~$(E-A)\la_2=\la_2-\la_1$ линейно зависимы. Значит, $2>r>0$, $r=1$, $\hn{P\sm P^A}=r+2=3$.

Таким образом, $A\la_1=\la_2\ne\pm\la_1$, $A\la_2=\la_1\ne\pm\la_2$, $\la_1,\la_2\notin P^A\cup P^{-A}$ и~$\hn{P\sm P^A}=3$. Отсюда
$P\sm P^A=\{\la,\la_1,\la_2\}$, где $\la\in P$ и~$A\la=\pm\la$.

Имеем $\la\in(P\sm P^A)\cap(P^A\cup P^{-A})=(P\sm P^A)\cap P^{-A}=P^{-A}\sm P^A=P^{-A}\sm\{0\}\subs\ggt^{-A}\sm\{0\}$. Поэтому
$\hn{P^{-A}}=\bn{P^{-A}\sm\{0\}}=\bn{(P\sm P^A)\cap P^{-A}}=1$. Ввиду соотношений $\la\in\ggt^{-A}\sm\{0\}$ и~$\rk(E-A)=r=1$, справедливо равенство
$(E-A)\ggt=\R\la$.

Заметим, что $0\ne\la_1-\la_2=(E-A)\la_1\in(E-A)\ggt=\R\la$, $\la_1-\la_2\in\R\la\sm\{0\}$. Отсюда
$\ha{P\sm P^A}=\ha{\la,\la_1,\la_2}=\ha{\la,\la_1}=\ha{\la,\la_2}=\ha{\la_1,\la_2}$. Значит, $P\sm P^A=\{\la,\la_1,\la_2\}\subs\ggt$\т неразложимое
множество, причём $\dim\ha{P\sm P^A}=\dim\ha{\la_1,\la_2}=2$.
\end{proof}

\begin{imp}\label{ppa} Предположим, что $g\in\Om$. Тогда подмножество $P\sm P^A\subs P\sm\{0\}$ содержится в~некоторой неразложимой компоненте множества
$P\sm\{0\}\subs\ggt$.
\end{imp}

\begin{proof} При $A=E$ доказывать нечего. В~случае же $A\ne E$ достаточно воспользоваться леммами \ref{apme} и~\ref{are}.
\end{proof}

\begin{imp}\label{adne} Предположим, что $g\in\Om$, а~также $A\ne E$. Тогда подмножество $P':=\bc{\la\in P\cln(E-g)V_{\la}\ne0}\subs P$ содержится
в~одной из неразложимых компонент множества $P\sm\{0\}\subs\ggt$. В~частности, $P'\subs P\sm\{0\}$.
\end{imp}

\begin{proof} Согласно следствию~\ref{aneo}, $P'\subs P\sm P^A$. Осталось применить следствие~\ref{ppa}.
\end{proof}

\begin{prop}\label{ade} Если $g\in\Om$ и~$A=E$, то все изотипные компоненты $V_{\la}\subs V$ \ter{$\la\in P$}, кроме, быть может, одной, содержатся
в~подпространстве $V^g\subs V$.
\end{prop}

\begin{proof} Поскольку $A=E$, имеем $gV_{\la}=V_{\la}$ ($\la\in P$) и~$g|_{V_{\la}}\in\GL_{\Cbb}(V_{\la})$ ($\la\in P\sm\{0\}$). Как следствие,
\equ{\begin{array}{c}
2\ge\om(g)=\rk(E-g)=\dim\br{(E-g)V}=\sums{\la\in P}\dim\br{(E-g)V_{\la}}=\\
=\dim\br{(E-g)V_0}+\sums{\sst{\la\in P\\\la\ne0}}2\cdot\dim_{\Cbb}\br{(E-g)V_{\la}}.
\end{array}\qedhere}
\end{proof}

\subsection{Разложение на компоненты}\label{decs}

Этот пункт посвящён доказательству теоремы~\ref{submain}.

Пусть $Q_1\sco Q_p\subs P$ ($p\in\N$)\т неразложимые компоненты множества $P\sm\{0\}\subs\ggt$. Имеем $P\sm\{0\}=\sqcupl{l=1}{p}Q_l\subs P$. Положим
$V_l:=\opluss{\la\in Q_l}V_{\la}\subs V$ ($l=1\sco p$).

В~записи~$V_0$ нижний индекс означает вес $0\in\ggt$\~ между тем, в~ряде случаев нам будет удобно понимать этот индекс и~как целое неотрицательное число.

Пространство~$V$ разлагается в~прямую сумму своих попарно ортогональных $G^0$\д инвариантных подпространств $V_l$, $l=0\sco p$. Далее, в~группе Ли~$G$
имеются подгруппы Ли $G_l:=\bc{g\in G\cln V_{l'}\subs V^g\fa l'\in\{0\sco p\}\sm\{l\}}$ ($l=0\sco p$) и~$\wt{G}:=G_0\sti G_p$. При этом для всякого
$l=0\sco p$ выполняется равенство $\wt{G}V_l=V_l$, а~представление $G_l\cln V_l$ точное. В~частности, $|G_0|<\bes$.

Нашей ближайшей целью является доказательство следующей теоремы.

\begin{theorem}\label{dec} Имеем $\wt{G}=G$.
\end{theorem}

Доказательству теоремы~\ref{dec} предпошлём несколько вспомогательных утверждений, которые в~дальнейшей части работы будут полезны и~сами по себе.

\begin{prop}\label{wtg0} Справедливо включение $G^0\subs\wt{G}$.
\end{prop}

\begin{proof} Для всякого $l=1\sco p$ подалгебра $\Lie G_l\subs\ggt$ есть не что иное как пересечение ядер всех весов
$\la\in Q_{l'}\subs P\subs\ggt=\ggt^*$, $l'\in\{1\sco p\}\sm\{l\}$. Поскольку $\ggt^*=\ha{P}=\oplusl{l=1}{p}\ha{Q_l}$, имеем
$\ggt=\oplusl{l=1}{p}\Lie G_l$, $\Lie\wt{G}=\ggt$, $\wt{G}\sups G^0$.
\end{proof}

\begin{prop} Справедливо включение $\Om\subs\wt{G}$.
\end{prop}

\begin{proof} Вытекает из следствия~\ref{adne} и~предложения~\ref{ade}.
\end{proof}

Пусть $v\in V$\т произвольный вектор, такой что $|G_v|<\bes$.

\begin{stm} Представление $G_v\cln N_v$ точное.
\end{stm}

\begin{proof} Если $g\in G_v$ и~$N_v\subs V^g$, то $\om(g)=\dim\br{(E-g)N_v}=0$, и, согласно утверждению~\ref{ane}, $\Ad(g)=E$, $\ggt v\subs V^g$,
$V^g\sups(\ggt v)\oplus N_v=V$, $g=E$.
\end{proof}

\begin{lemma}\label{wtgv} Имеем $G_v\subs\wt{G}$. Кроме того, существуют разложения $G_v=H_0\times H_1\sti H_k$ и~$N_v=W_0\oplus W_1\sop W_k$
\ter{$k\in\Z_{\ge0}$}, удовлетворяющие следующим условиям\:
\begin{nums}{-1}
\item\label{sys} подпространства $W_0,W_1\sco W_k\subs N_v$ попарно ортогональны и~$G_v$\д инвариантны\~
\item\label{res} для любых $i,j=0\sco k$ линейная группа $(H_i)|_{W_j}\subs\Or(W_j)$ тривиальна при $i\ne j$, порождена псевдоотражениями при $i=j=0$
и~изоморфна группе Пуанкаре при $i=j>0$ \ter{в~частности, $\dim W_j=4$ для всякого $j=1\sco k$}\~
\item\label{coin} если $\ha{G_v\cap\Om}\ne G_v$, то $k\ge1$, а~если $[G_v,G_v]\ne G_v$, то $\dim N_v\ge4k+2$\~
\item\label{iso} для любого $j=1\sco k$ найдётся вес $\la\in P$, такой что $V_{\la}\sups W_j$ и~$V_{\la}^{\perp}\subs V^{H_j}$.
\end{nums}
\end{lemma}

\begin{proof} В~силу теоремы~\ref{slice}, $N_v/G_v$\т гомологическое многообразие. Далее, применяя к~точному представлению $G_v\cln N_v$
теорему~\ref{lang}, получаем, что существуют разложения $G_v=H_0\times H_1\sti H_k$ и~$N_v=W_0\oplus W_1\sop W_k$, $k\in\Z_{\ge0}$, удовлетворяющие
условиям \ref{sys} и~\ref{res}. При этом $H_0=\ha{H_0\cap\Om}$, а~группа Пуанкаре совпадает со своим коммутантом, вследствие чего условие~\ref{coin}
также выполняется.

Пусть $j\in\{1\sco k\}$\т произвольное число.

Если $h\in H_j\sm\{E\}$ и~$h^5=E$, то $\om(h)=\dim\br{(E-h)N_v}=\dim W_j=4$ и, кроме того, $\om(h^5)=\om(E)=0$, откуда $h\in\Om'$, что вместе
с~леммой~\ref{omp} влечёт равенство $\Ad(h)=E$. Как известно, группа Пуанкаре порождается своими элементами порядка~$5$. Поэтому $\Ad(H_j)=\{E\}$. Отсюда
следует, что, во-первых, $H_jV_{\la}=V_{\la}$ для всякого $\la\in P$, а~во-вторых, что $\ggt v\subs V^{H_j}$. Значит,
$V^{H_j}=\ggt v\oplus N_v^{H_j}=\ggt v\oplus(N_v\cap W_j^{\perp})=W_j^{\perp}$. Тем самым нами установлено, что
\begin{itemize}
\item изотипные компоненты представления $H_j\cln V$ суть в~точности подпространства $W_j$ и~$V^{H_j}=W_j^{\perp}$ пространства~$V$, причём
представление $H_j\cln W_j$ неприводимо\~
\item пространство~$V$ разлагается в~прямую сумму своих попарно ортогональных $H_j$\д инвариантных подпространств $V_{\la}$, $\la\in P$.
\end{itemize}
Поэтому найдётся вес $\la\in P$, для которого $V_{\la}\sups W_j$ и~$V_{\la}^{\perp}\subs V^{H_j}$. Значит, $H_j\subs\wt{G}$.

Итак, $H_1\sco H_k\subs\wt{G}$. Кроме того, $H_0=\ha{H_0\cap\Om}\subs\ha{\Om}\subs\wt{G}$. Отсюда $G_v\subs\wt{G}$.
\end{proof}

\begin{imp}\label{ngen} Допустим, что $\ha{G_v\cap\Om}\ne G_v$. Тогда найдётся вес $\la\in P$, такой что $\dim V_{\la}\ge4$. Кроме того, если
$[G_v,G_v]\ne G_v$, то $\dim N_v\ge6$.
\end{imp}

Теперь утверждение теоремы~\ref{dec} вытекает непосредственно из теоремы~\ref{fst}, предложения~\ref{wtg0} и~леммы~\ref{wtgv}.

Имеем $G=\wt{G}=G_0\sti G_p$. Поэтому $V/G\cong(V_0/G_0)\sti(V_p/G_p)$, и, согласно лемме~\ref{prop}, каждый из факторов $V_l/G_l$, $l=0\sco p$,
является гомологическим многообразием. Далее, $Q_1\sco Q_p\subs P\sm\{0\}$, вследствие чего для любого $l=1\sco p$ множество весов представления
$G_l\cln V_l$ неразложимо, $2$\д устойчиво и~не содержит нулей.

Тем самым нами полностью доказана теорема~\ref{submain}.

\subsection{Неразложимый случай}\label{indec}

Этот пункт посвящён доказательствам импликаций $\text{\ref{ho}}\Ra\text{\ref{crit}}$ в~теоремах \ref{main} и~\ref{main1}.

Мы по-прежнему считаем, что $V/G$\т гомологическое многообразие, а~множество $P\subs\ggt$ является $2$\д устойчивым. Кроме того, будем предполагать, что
множество $P\subs\ggt$ неразложимо и~не содержит нулей.

Согласно лемме~\ref{hp3}, $\Ad(G)\ne\{E\}$.

В~силу леммы~\ref{exa}, $\hn{P}=m+2$. Ввиду $2$\д устойчивости множества $P\subs\ggt$, любые его векторы в~количестве не более~$m$ линейно независимы.
В~частности, при $m\ge2$ данное множество не содержит кратных векторов.

\subsubsection{Доказательство теоремы~\ref{main}}

Вначале докажем импликацию $\text{\ref{ho}}\Ra\text{\ref{crit}}$ в~теореме~\ref{main}.

Предположим, что $m\ge2$.

Требуется доказать, что выполнены условия \ref{pm2}---\ref{finst} из формулировки теоремы~\ref{main}.

Как уже отмечалось, $\hn{P}=m+2$. Пространство~$V$ разлагается в~прямую сумму попарно ортогональных двумерных неприводимых $G^0$\д инвариантных
подпространств $W_1\sco W_{m+2}\subs V$. При этом множество $P\subs\ggt$ не содержит кратных векторов. Значит,
\begin{itemize}
\item подпространства $W_1\sco W_{m+2}\subs V$ являются изотипными компонентами представления $G^0\cln V$ и~переставляются группой~$G$\~
\item для всякого $\la\in P$ имеем $\dim V_{\la}=2$.
\end{itemize}
Согласно следствию~\ref{ngen}, если $v\in V$ и~$|G_v|<\bes$, то $G_v=\ha{G_v\cap\Om}$. Теперь, пользуясь теоремой~\ref{fst}, получаем, что
$G=\ha{G^0\cup\Om}$, $\Ad(G)=\ba{\Ad(\Om)}$.

Таким образом, мы уже доказали, что условия \ref{pm2} и~\ref{finst} выполняются.

Пусть $g\in\Om$\т произвольный элемент, такой что $A:=\Ad(g)\ne\pm E$.

В~силу леммы~\ref{apme}, $P\ne P^A\cup P^{-A}$. Применяя лемму~\ref{are}, получаем, что $\hn{P\sm P^A}=3$, $\dim\ha{P\sm P^A}=2$ и~$\hn{P^{-A}}=1$. Мы
видим, что множество $P\subs\ggt$ содержит три линейно зависимых вектора. Значит, $3>m\ge2$, $m=2$, $\hn{P}=4$, $\hn{P^A}=\hn{P}-\hn{P\sm P^A}=1$.

Допустим, что $m>2$.

Из вышесказанного следует, что $\Ad(\Om)\subs\{\pm E\}$, $\Ad(G)=\ba{\Ad(\Om)}\subs\{\pm E\}$, откуда $GW_j=W_j\fa j=1\sco m+2$. Далее, поскольку
$\Ad(G)\ne\{E\}$, имеем $\Ad(G)=\{\pm E\}$, и, таким образом, все условия \ref{pm2}---\ref{finst} выполняются.

Теперь предположим, что $m=2$.

Имеем $\hn{P}=4$, $P=\{\la_1,\la_2,\la_3,\la_4\}\subs\ggt$, $\la_1,\la_2,\la_3,\la_4\in\ggt\sm\{0\}$. Любые два вектора множества $P\subs\ggt$ линейно
независимы, и, значит, прямые $\R\la_j\subs\ggt$, $j=1,2,3,4$, попарно различны. Будем считать, что
\eqn{\label{pair}
\begin{array}{lcc}
\fa i,j\in\{1,2,3,4\}&\quad\quad&\br{(\la_i,\la_j)=0}\Ra\Br{\br{\{i,j\}=\{1,2\}}\lor\br{\{i,j\}=\{3,4\}}};\\
\fa j\in\{1,2,3,4\}&\quad\quad&V_{\la_j}=W_j\subs V
\end{array}}
(этого можно добиться путём надлежащих перенумераций).

Рассмотрим произвольный элемент $g\in\Om$.

Покажем, что $g(W_1\oplus W_2)=W_1\oplus W_2$.

При $A:=\Ad(g)=\pm E$ доказывать нечего.

Допустим, что $A\ne\pm E$. Тогда $\hn{P^A}=\hn{P^{-A}}=1$. Следовательно, найдутся числа $i,j\in\{1,2,3,4\}$, такие что $A\la_i=\la_i$ и~$A\la_j=-\la_j$.
Очевидно, что $(\la_i,\la_j)=0$, $gV_{\la_i}=V_{\la_i}$ и~$gV_{\la_j}=V_{\la_j}$. В~силу~\eqref{pair}, $g(W_1\oplus W_2)=W_1\oplus W_2$.

Тем самым мы установили, что $g(W_1\oplus W_2)=W_1\oplus W_2$ для всякого $g\in\Om$. При этом $G=\ha{G^0\cup\Om}$, откуда
$G(W_1\oplus W_2)=W_1\oplus W_2$.

Таким образом, условия \ref{pm2}, \ref{oplu} и~\ref{finst} выполняются.

Осталось проверить условие~\ref{adj}.

Достаточно доказать, что $-E\in\Ad(G)$.

Допустим, что $-E\notin\Ad(G)$.

Согласно лемме~\ref{hp3}, группа $\Ad(G)\subs\Or(\ggt)$ содержит отражение относительно каждой из четырёх попарно различных прямых $\R\la_j\subs\ggt$,
$j=1,2,3,4$. Значит, $\bgm{\Ad(G)}>4$.

Поскольку $G(W_1\oplus W_2)=W_1\oplus W_2$, для любого $g\in G$ имеем $g^2W_j=W_j$ ($j=1,2,3,4$), $P\subs\Ker\br{E-\Ad(g^2)}\cup\Ker\br{E+\Ad(g^2)}$, что
вместе с~неразложимостью множества $P\subs\ggt$ и~соотношением $-E\notin\Ad(G)$ влечёт равенство $\br{\Ad(g)}^2=\Ad(g^2)=E$. Мы видим, что все операторы
группы $\Ad(G)\subs\Or(\ggt)$ инволютивны. Поэтому $\Bm{\br{\Ad(G)}\cap\br{\SO(\ggt)}}\le2$, $\bgm{\Ad(G)}\le4$, что противоречит неравенству
$\bgm{\Ad(G)}>4$.

Следовательно, $-E\in\Ad(G)$, и, таким образом, все условия \ref{pm2}---\ref{finst} выполняются.

Тем самым теорема~\ref{main} полностью доказана.

\subsubsection{Доказательство теоремы~\ref{main1}}

Теперь докажем импликацию $\text{\ref{ho1}}\Ra\text{\ref{crit1}}$ в~теореме~\ref{main1}.

Предположим, что $m=1$, а~группа $G\subs\Or(V)$ не содержит комплексных отражений.

Требуется доказать, что $\dim_{\Cbb}V=\hn{P}=3$, $\Ad(G)=\{\pm E\}$, $G=\ha{\Om}$, а~представление $G\cln V$ приводимо.

Как уже отмечалось, $\dim_{\Cbb}V=\hn{P}=m+2=3$. Кроме того, $\Ad(G)\ne\{E\}$, откуда $\Ad(G)=\{\pm E\}$. Если $v\in V$, $|G_v|<\bes$
и~$G_v\ne\ha{G_v\cap\Om}$, то $\dim N_v=\dim V-\dim G=5$ и, согласно следствию~\ref{ngen}, $[G_v,G_v]=G_v\ne\{E\}$, $G_v\subs[G,G]\subs\Ker\Ad$, что,
в~частности, влечёт неразрешимость группы $\Ker\Ad\subs G$.

Допустим, что представление $G\cln V$ неприводимо.

Поскольку группа $G\subs\Or(V)$ неприводима и~не содержит комплексных отражений, имеем $G^0=\T E\subs\GL_{\Cbb}(V)$ и~$\Ker\Ad=G^0K\subs G$, где
$K:=\Ker\Ad\cap\SL_{\Cbb}(V)\subs\GL_{\Cbb}(V)$\т конечная неприводимая импримитивная комплексная линейная группа (см.~\cite[\Ss7]{My1}).

Комплексное пространство~$V$ разлагается в~прямую сумму трёх одномерных комплексных подпространств, переставляемых группой $K\subs\GL_{\Cbb}(V)$,
а~значит, и~группой $\Ker\Ad=G^0K\subs\GL_{\Cbb}(V)$. Поэтому существует гомоморфизм $\Ker\Ad\to S_3$ с~коммутативным ядром. Группа~$S_3$ разрешима\~ то
же можно сказать и~о~группе $\Ker\Ad\subs G$.

Следовательно, если $v\in V$ и~$|G_v|<\bes$, то $G_v=\ha{G_v\cap\Om}$. Пользуясь теоремой~\ref{fst}, получаем, что $G=\ha{G^0\cup\Om}$.

Таким образом, $\dim G=1$, $0\notin P$, $\dim_{\Cbb}V=\hn{P}=3$, $G=\ha{G^0\cup\Om}$, $G_v=\ha{G_v\cap\Om}$ ($v\in V$, $|G_v|<\bes$), а~группа
$G\subs\Or(V)$ неприводима и~не содержит комплексных отражений. Данная ситуация невозможна (см.~\cite[\Ss7]{My1}, рассуждения в~точности повторяют
доказательство леммы~7.2).

Полученное противоречие показывает, что представление $G\cln V$ приводимо.

Осталось доказать, что $G=\ha{\Om}$.

Приводимое представление $G\cln V$ обладает двумерным инвариантным подпространством. Значит, найдётся вектор $v\in V\sm\{0\}$, для которого $Gv=G^0v$.
Имеем $|G_v|<\bes$ и~$G=G^0G_v$. Если $G_v\ne\ha{G_v\cap\Om}$, то $G_v\subs\Ker\Ad$, $G=G^0G_v\subs\Ker\Ad$, что противоречит равенству
$\Ad(G)=\{\pm E\}$. Отсюда $G_v=\ha{G_v\cap\Om}$, $G=G^0G_v\subs\ha{G^0\cup\Om}$.

Мы видим, что $G=\ha{G^0\cup\Om}$. Поскольку $\Ad(G)=\{\pm E\}$, существует элемент $g\in\Om$, такой что $\Ad(g)=-E$. Далее, в~группе~$G$
каждый элемент подмножества $G^0g$ сопряжён элементу $g\in\Om$ и~потому принадлежит подмножеству~$\Om$. Отсюда $\ha{\Om}\sups G^0$,
$\ha{\Om}\sups G^0\cup\Om$, $\ha{\Om}\sups\ha{G^0\cup\Om}=G$, $G=\ha{\Om}$.

Тем самым теорема~\ref{main1} полностью доказана.

\end{document}